\documentclass[12pt,reqno, a4paper]{amsart}

\usepackage[utf8]{inputenc}
\usepackage[margin=3.1cm]{geometry}
\usepackage{graphicx}
\usepackage{todonotes} 
\usepackage{subcaption}
\usepackage{mathtools}
\usepackage{hyperref}
\usepackage{MnSymbol}
\usepackage{wrapfig}

\usepackage{subcaption}

\usepackage{amsmath, amsthm}
\usepackage{esint}
\usepackage{verbatim}
\usepackage{dsfont}
\usepackage{enumitem}
\usepackage{helvet}
\numberwithin{equation}{section}
\usepackage[normalem]{ulem}

\DeclareMathOperator{\dist}{dist}

\DeclareMathOperator{\dive}{div}

\DeclareMathOperator{\supp}{supp}

\renewcommand{\hat}{\widehat}
\renewcommand{\tilde}{\widetilde}
\newcommand{\Lip}{\mathrm{Lip}}

\newcommand{\ip}[1]{\langle {#1}\rangle}

\newcommand{\Lm}{\mathcal{L}} 
\newcommand{\Hm}{\mathcal{H}} 
\newcommand{\N}{\mathbb{N}} 
\newcommand{\Z}{\mathbb{Z}} 
\newcommand{\R}{\mathbb{R}} 
\newcommand{\calA}{\mathcal{A}}
\newcommand{\calAo}{\mathcal{A}_\omega}
\newcommand{\calAoe}{\mathcal{A}_{\omega,\eps}}
\newcommand{\calAe}{\mathcal{A}_{\eps}}
\newcommand{\cX}{\mathcal X}
\newcommand{\cXo}{\mathcal X_\omega}
\newcommand{\cXoe}{{\mathcal X_{\omega,\eps}}}
\newcommand{\cXe}{{\mathcal X_{\eps}}}
\newcommand{\cE}{\mathcal E}
\newcommand{\cEo}{\mathcal E_\omega}
\newcommand{\cEoe}{\mathcal E_{\omega,\eps}}
\newcommand{\cEe}{\mathcal E_{\eps}}

\newcommand{\del}{\partial}

\newcommand{\M}{\mathcal{M}}

\newcommand{\weakstar}{\stackrel{*}{\rightharpoonup}}
\newcommand{\eps}{\varepsilon}

\newcommand{\KR}{\mathrm{KR}}

\newtheorem{theorem}{Theorem}[section]
\newtheorem{prop}[theorem]{Proposition}

\newtheorem{lemma}[theorem]{Lemma}
\newtheorem{defi}[theorem]{Definition}
\newtheorem{cor}[theorem]{Corollary}
\newtheorem{example}[theorem]{Example}
\newtheorem{assumption}[theorem]{Assumption}

\theoremstyle{remark}
\newtheorem{rem}[theorem]{Remark}

\def\calA{{\mathcal A}}  
  \def\calF{{\mathcal F}}
  
\def\calJ{{\mathcal J}}  
\def\calM{{\mathcal M}}  
\def\calP{{\mathcal P}}

\begin{document}

\title{Stochastic homogenization of dynamical discrete optimal transport}

\author{Peter Gladbach}
\author{Eva Kopfer}
\thanks{The second author gratefully acknowledges support by the German
Research Foundation through the Hausdorff Center for Mathematics and the Collaborative Research Center 1060.}

\address{Institut f\"ur Angewandte Mathematik \\ Rheinische Friedrich-Wilhelms-Universit\"at Bonn \\ Endenicher Allee 60 \\ 53115 Bonn, Germany}
	\email{gladbach@iam.uni-bonn.de} \email{eva.kopfer@iam.uni-bonn.de}

\date{}

\begin{abstract}
    The aim of this paper is to examine the large-scale behavior of dynamical optimal transport on stationary random graphs embedded in $\R^n$. Our central theorem is a stochastic homogenization result that characterizes the effective behavior of the discrete problems in terms of a continuous optimal transport problem, where the homogenized energy density results from the geometry of the discrete graph.
\end{abstract}

	\maketitle

\tableofcontents

\section{Introduction}

Over the past few decades, optimal transport has emerged as a vibrant research area at the intersection of analysis, probability, and geometry. Significant research activity has occurred in both pure mathematics and applied areas \cite{villani2021topics,santambrogio2015optimal,peyre2019computational}. A central focus in this field is the $2$-Kantorovich distance $W_2$, which plays a crucial role in non-smooth geometry \cite{sturm2006geometry,lott2009ricci} and the theory of dissipative PDEs such as the heat, Fokker-Planck, and porous-medium equations \cite{jordan1998variational,otto2001geometry}.

The Benamou–Brenier formula \cite{benamou2000computational} is a pivotal result in continuous settings, demonstrating the equivalence between static and dynamical optimal transport. It asserts that the classical Monge–Kantorovich problem, which minimizes a linear cost functional over couplings of given probability measures $\mu_0$ and $\mu_1$, is equivalent to a dynamical transport problem, where an action functional is minimized over all solutions to the continuity equation connecting $\mu_0$ and $\mu_1$. For an overview of the theory and its applications, we refer to the monographs \cite{villani2021topics,ag,santambrogio2015optimal}.

In discrete settings, however, this equivalence between static and dynamical optimal transport breaks down since the space of probability measures contains no non-trivial curves of finite $W_2$-length. The dynamical optimal transport problem on graphs \cite{maas2011gradient,mielke2011gradient} is essential for gradient flow formulations of dissipative evolution equations, such as the heat equation \cite{maas2011gradient,mielke2011gradient}, porous medium \cite{erbar2012gradient} and McKean-Vlasov \cite{erbar2016gradient}, as well as for discrete Ricci curvature \cite{erbar2012ricci} and functional inequalities \cite{mielke2013geodesic,erbar2018poincare}. Consequently, analyzing the discrete-to-continuum limit of dynamical optimal transport in various settings is a significant problem. This has been addressed in several articles, which we briefly recall here:

 Gigli and Maas show in \cite{gigli2013gromov} the convergence of the discrete action on $\eps\Z^n$ with nearest neighbors to the $W_2$-action.

Gladbach et al.\ show in \cite{gladbach2023homogenisation} the convergence to a homogenized action (not necessarily that of $W_2$) for $\eps \Z^n$-periodic graphs.

 Gladbach, Kopfer, and Maas show in \cite{gladbach2020scaling} that in the finite-volume setting, where points are neighbors if and only their Voronoi cells touch, the discrete action converges to that of $W_2$ if and only if the Voronoi cells satisfy a geometric isotropy condition, which unfortunately happens almost never under random shifts.

On the other hand, Garcia Trillos shows in \cite{garcia2020gromov} the convergence to the $W_2$ action for random point clouds provided that every point is neighbors with all other points in a large enough ball, so that the number of neighbors grows to infinity as the spacing decreases.

Our main result fits in between these: We show the stochastic homogenization of dynamical transport actions on certain random graphs to a continuous anisotropic action. Unlike in \cite{garcia2020gromov} the vertex degrees remain bounded. Unlike in \cite{gladbach2023homogenisation} the graph is not necessarily periodic.

Examples of random graphs are road networks on a landmass \cite{hackl2016generation,sohouenou2020using}, blood vessels forming in a tissue through angiogenesis \cite{cha2005vascular}, and water channels in porous limestone \cite{tang2021fluid}.





Stochastic homogenization is a mathematical theory used to study the macroscopic behavior of differential equations with random spatially heterogeneous coefficients, which are regarded as a realization of a random field. The goal in homogenization is to find an effective homogeneous equation that approximates the behavior of the original PDE with random coefficients at large scales.

In the literature, the random field is often assumed to be stationary (its statistical properties do not change over space) and ergodic (spatial averages converge to ensemble averages) in order to use ergodic theorems, which originated with Birkhoff  \cite{birkhoff1931proof}, and remains an extremely active research field, see e.g. \cite{armstrong2018stochastic,armstrong2016quantitative,berlyand2017two,gloria2015quantification} and references therein.

In the variational setting and specifically in this article, homogenization results can be expressed in terms of $\Gamma$-convergence. The $\Gamma$-limit of a sequence of functionals is the largest asymptotic lower bound.  This notion was introduced by De Giorgi \cite{de1975tipo} as the natural type of convergence for minimization problems. We refer to \cite{braides2002gamma,dal2012introduction} for detailed introductions to $\Gamma$-convergence.

In order to show $\Gamma$-convergence, we employ Kingman's subadditive ergodic theorem \cite{kingman1968ergodic}. The first application by Dal Maso and Modica \cite{dal1986nonlinear} characterizes the homogeneous limit functional of stationary random integral functionals
\[
F_{\omega,\eps}( u) = \int_{\R^n} f_\omega( x/\eps, \nabla u(x))\,dx,
\]
where $f_\omega: \R^n\times \R^n\to \R$ is convex in $\nabla u$ with $p$-growth, $1<p<\infty$.

This result was recently expanded to the nonconvex $p=1$ case by Cagnetti et al. in \cite{cagnetti2022global}, yielding a limit functional which is finite exactly for functions of bounded variation.

Alicandro, Cicalese, and Gloria study in \cite{alicandro2011integral} the $\Gamma$-limit of random discrete gradient functionals on random graphs of the form
\[
F_{\omega,\eps}(u) = \eps^n \sum_{(x,y)\in \cEo} f_\omega\left(x,y, \frac{u(y)-u(x)}{\eps}\right), 
\]
where again $f_\omega: \cEo \times \R \to \R$ has $p$-growth, $1<p<\infty$, and the limit functional is of the same form as in \cite{dal1986nonlinear}. Both \cite{dal1986nonlinear} and \cite{alicandro2011integral} treat functionals acting on gradients, i.e. curl-free vector fields. In the transport setting, gradients are replaced by mass-flux pairs $(\rho,j)$ solving the continuity equation $\partial_t \rho + \dive j = 0$, which can be interpreted as zero space-time divergence.

Gladbach, Maas, and Portinale show in the recent article \cite{gladbach2024stochastic} the stochastic homogenization of flow costs on random graphs of the type
\[
F_{\omega,\eps}(J) = \eps^n \sum_{(x,y)\in \cEo} f_\omega\left(x,y, \frac{J(x,y)}{\eps^{n-1}}\right), 
\]
where $f_\omega : \cEo \times \R \to \R$ has $1$-growth but may be nonconvex, and in fact $J:\cEo \to \R$ is a skew-symmetric function with a divergence constraint. We adapt many of the methods therein to this article. In fact, the most straightforward example of such a cost is the total variation $f_\omega\left(x,y, J\right) = |y-x||J|$. Decomposing $\sum_{y\, :\, (x,y)\in\cEo}J(x,y)$ into its positive and negative part, by the flow decomposition theorem \cite{ford1962flows} we get for each signed measure $m=m_+-m_-$ of total mass $0$
\[
\min\left\{F_{\omega}(J)\, :\, \sum_{y:(x,y)\in\cEo}J(x,y) = m(x)\ \forall x\in \cXo\right\} = W_1^{(\cXo,\cEo)}(m_-, m_+).
\]

By contrast, this article deals not with the graph-based earth-mover's distance $W_1^{(\cXo,\cEo)}$ but the graph-based $W_2$-action, first introduced in \cite{maas2011gradient, mielke2011gradient}, which necessitates the inclusion of a time-variable as well as intermediate mass distributions. Nonetheless, we adapt several of the methods developed in \cite{gladbach2024stochastic} to the time-dependent case.

First, our Assumption \ref{ass: graph} on the stationary random graph is taken directly from \cite{gladbach2024stochastic}. A similar assumption also occurs in \cite{alicandro2011integral}.

Second, the boundary values in the cell-formula \eqref{eq: rep} are the same as in \cite{gladbach2024stochastic}.

Third, the blow-up procedure with tangent measures used to prove the lower bound in Section \ref{sec: lower} is strongly inspired by the one in \cite{gladbach2024stochastic}, and has a lot of parallels with the one in \cite{cagnetti2022global}.

Neither of the cited results are directly applicable to the discrete dynamic optimal transport problem. The Kantorovich transport cost $\frac{|j|^2}{\rho}$, while convex, has nonstandard growth at infinity, requiring some additional considerations, which we detail below.


\subsection{Main result}
In this paper we consider a random graph $(\cXo,\cEo)_{\omega\in \Omega}$ in $\R^n$ with random weights $\sigma_\omega\colon \cEo\to [\lambda,\Lambda]$, with $0<\lambda<\Lambda$, where $(\Omega,\calF,\mathbb{P})$ is a probability space. We assume that the graph and the weights are stationary with respect to translations in $\Z^n$. The object of interest is the action functional $\calAo(m,J)$ defined for curves of probability measures $m\in W^{1,1}((0,T);\calP(\cXo))$ and curves of skew-symmetric vector fields $J\in L^1((0,T);\R_a^{\cEo})$ satisfying the discrete continuity equation
\begin{align*}
    \del_tm_t(x)+\sum_{y\sim x}J_t(x,y)=0 \quad \forall x\in\cXo.
\end{align*}
This action functional is given by
\begin{align*}
\calAo(m,J)=\int_0^TF_\omega(m_t,J_t)\, dt,
\end{align*}
where $F_\omega$ is the discrete kinetic energy
\begin{align*}
F_\omega(m_t,J_t)=\sum_{(x,y)\in\cEo}\sigma_\omega(x,y)\frac{|x-y|^2|J_t(x,y)|^2}{\theta^\omega_{xy}(m_t(x),m_t(y))}.
\end{align*}
Here, $\theta^\omega_{xy}$ is a family of random mean functions. We refer to Definitions \ref{defi: mean} and \ref{defi: action} for precise formulations.
The functional under consideration is the discrete version of the action used in the Benamou-Brenier formula of the $L^2$-Kantorovich distance \cite{benamou2000computational} and variants of this have been studied in e.g. \cite{maas2011gradient,erbar2012ricci,gigli2013gromov,gladbach2020scaling,gladbach2023homogenisation,mielke2011gradient}. We want to understand the effective behavior by means of stochastic homogenization. In order to do so, we consider the rescaled random graph $(\cXoe,\cEoe)=(\eps\cXo,\eps\cEo)$.
The rescaled action functional $\calAoe$ is then obtained through
\begin{align*} 
\calAoe(m,J) = \eps^2 \calAo(m(\cdot/\eps),J(\cdot/\eps)),
\end{align*}
since $|x-y|^2$ scales as $\eps^2$. We refer to Section \ref{sec: rescaled graph} for details.

We will show that under Assumption \ref{ass: graph}, the rescaled action functionals $\calAoe$ $\Gamma$-converge to the homogenized action functional $A_\omega$, as $\eps\to0$.

The homogenized action $A_\omega$ is defined on finite Radon measures in space-time $(\rho,j)\in\M((0,T)\times \R^n;[0,\infty)\times \R^n)$ with finite action, see Definition \ref{defi: finite action}, which can be disintegrated into curves of measures in space: 
\begin{align*}
    A_\omega(\rho,j)=\int_0^T\int_{\R^n}f_\omega\left(\frac{dj_t}{d\rho_t}\right)\, d\rho_t\, dt.
\end{align*}
Here, the homogenized energy density $f_\omega : \R^n \to [0,\infty)$ is explicitly given by the cell formula \eqref{eq: cell}, and is $2$-homogeneous and convex. The Kantorovich action corresponds to the special isotropic case $f_\omega(v)=|v|^2$ by the Benamou-Brenier formula.

In order to get convergence of our rescaled action functionals, we  require additional assumptions on the random graph, which are summarized in Assumption \ref{ass: graph}. 


The main result reads as follows. Note that we embed measures $m$ on a graph $\cX$ into measures on $\R^n$ in the natural way and denote by $\iota J$ the embedding of skew-symmetric vector fields $J\colon\cE\to \R$ into vector-valued measures on $\R^n$ made precise in Definition \ref{defi: iota}.

\begin{theorem}\label{thm: main}
Let $(\cXo,\cEo)_{\omega\in\Omega}$ be a random graph in $\R^n$ with random weights $(\sigma_\omega\colon \cEo\to[\lambda,\Lambda])_{\omega\in\Omega}$ and random means $(\theta^\omega)_{\omega\in\Omega}$ with $0<\lambda<\Lambda$, satisfying the graph Assumption \ref{ass: graph} almost surely, such that $(\cXo,\cEo,\sigma_\omega,\theta_\omega)$ is stationary.

For every $\eps>0$ let $(\cXoe,\cEoe)$ be the rescaled random graph and $\calAoe$ be the rescaled action functional defined in Section \ref{sec: rescaled graph}. Let $A_\omega$ be the homogenized action functional given in \eqref{eq: hom action}.
Then we have almost surely $\calAoe\stackrel{\Gamma}{\longrightarrow}_{\eps\to 0} A_\omega$ in the sense that the following two conditions hold almost surely:
\begin{enumerate}
    \item Lower bound: For every $m^\eps \weakstar \rho$ narrowly in $\M_+((0,T)\times \R^n)$ and $\iota J^\eps \weakstar j$ narrowly in $\M((0,T)\times \R^n;\R^n)$,
\begin{equation*}
    \liminf_{\eps\to 0}  \calAoe(m^\eps,J^\eps) \geq  A_\omega(\rho,j).
\end{equation*}
    \item Upper bound:
    For every curve $(\rho,j)$ there exist sequences $m^\eps \weakstar \rho$ narrowly in $\M_+((0,T)\times \R^n)$ \text{ and } $\iota J^\eps \weakstar j$  narrowly in $\M((0,T)\times \R^n;\R^n)$ such that
\begin{equation*}
    \limsup_{\eps\to 0}  \calAoe(m^\eps,J^\eps) \leq  A_\omega(\rho,j).
\end{equation*}
In addition, $\max_{t\in[0,1]}W_2(m^\eps_t,\rho_t) \to 0$.
\end{enumerate}

\end{theorem}


     We emphasize that in this article, we allow for nonoptimal fluxes $j$. 

Since we also show compactness of bounded action sequences in Section \ref{sec: compactness}, we obtain the following corollary, which we prove in Section \ref{sec: proofc}.

\begin{cor}\label{cor: main}
    Let $\rho_0,\rho_1 \in \calP(\R^n)$ with $W_2(\rho_0,\rho_1)<\infty$. Let  $\bar m^\eps_0, \bar m^\eps_1\in \calP(\cXoe)$ with $W_2(\bar m^\eps_0, \rho_0) \to 0$, $W_2(\bar m^\eps_1, \rho_1) \to 0$.
        
     Let $(m^\eps, J^\eps)$ be  minimizers of the discrete action on $[0,1]$ with $m^\eps_0 = \bar m^\eps_0$, $m^\eps_1 = \bar m^\eps_1$. Then the curves have convergent subsequences $m^\eps \weakstar \rho$, $\iota J^\eps \weakstar j$, and
    \begin{equation*}
        C_\omega(\rho_0,\rho_1) = A_\omega(\rho,j) = \lim_{\eps \to 0} \calAoe(m^\eps, J^\eps).
    \end{equation*}

    Here the homogenized Wasserstein distance $C_\omega$ can be characterized as either the minimal action
    \begin{equation*}
        C_\omega(\rho_0,\rho_1)= \min_{(\rho,j)}A_\omega(\rho,j)
    \end{equation*}
    or as the minimum Monge-Kantorovich cost
    \begin{equation*}
    C_\omega(\rho_0,\rho_1)= \min_{\gamma\in \mathrm{Cpl}(\rho_0,\rho_1)}\langle \gamma, f_\omega(y-x)\rangle,
    \end{equation*}
     where $\mathrm{Cpl}(\rho_0,\rho_1)$ is the set of all couplings of $\rho_0$ and $\rho_1$, i.e. all measures $\gamma\in \calP(\R^n\times\R^n)$ with marginals $\rho_0$ and $\rho_1$.
 
\end{cor}
By Lemma \ref{lemma: f}, the transport cost $C_\omega$ is almost surely comparable to $W_2^2(\rho_0,\rho_1)$.

\section{Preliminaries}
\subsection{Assumption on the graph geometry}
In the following, let $(\cX,\cE)$ be a countably infinite graph in $\R^n$.  We assume the following (cf. \cite[(G1)--(G3)]{gladbach2024stochastic}.

\begin{assumption}\label{ass: graph}
There is a constant $R>0$ such that:
\begin{enumerate}
    \item\label{path} For all $x,y\in\cX$ there exists a path $P$ in $(\cX,\cE)$ connecting $x,y$ with Euclidean length $L(P)\leq R(|x-y|+1)$.
    \item\label{ball} For all $x\in\R^n$, we have $\cX\cap B(x,R)\neq \emptyset$.
    \item\label{edge} The maximum edge length is bounded: $|x-y|\leq R$ for all $(x,y)\in \cE$.
    \item\label{degree} The maximum degree $\max_{x\in \cX} \deg(x)$ is finite.
\end{enumerate}
\end{assumption}

We briefly comment on the different points of this assumption:

The first point ensures that the graph is connected and that the graph distance is equivalent to Euclidean distance.

The second point ensures that the graph is dense enough. In particular, graphs with arbitrarily large holes are excluded.

The third point ensures that the limit action is purely local and no long-range transport persists.

The final point affects the scaling law of the action. By contrast, in \cite{garcia2020gromov}, the minimum vertex degree tends to infinity, allowing much lower actions. See Section \ref{sec: counterexamples} for an extended discussion of the correct scaling law under unbounded vertex degrees. The final point is not necessary in the $W_1$-case studied in \cite{gladbach2024stochastic}.

Assumptions \eqref{edge} and \eqref{degree} can potentially be weakened by introducing appropriate moment bounds on edge lengths and vertex degrees. However, these generalizations are beyond the scope of this article.

\subsection{Action functional}\label{sec: action}

Let $\calP(\cX)$ denote the space of probability measures on $\cX$ and $\R_a^{\cE}$ denote all skew-symmetric vector fields $J\colon \cE\to\R$.
For $m\in\calP(\cX)$ and $J\in\R_a^{\cE}$ we define the 
energy
\begin{equation}\label{eq: energy}
    F( m,J)=\sum_{(x,y)\in \cE}\sigma(x,y)\frac{|x-y|^2| J(x,y)|^2}{\theta_{xy}(m(x),m(y))}.
\end{equation}

Here, $\sigma\colon \cE\to [\lambda,\Lambda]$ are weights with $0<\lambda<\Lambda$ and
$(\theta_{xy})_{(x,y)\in \cE}$ is a family of mean functions which we shall define next.

\begin{defi}\label{defi: mean}
We say that $(\theta_{xy})_{x,y\in \cX}$ is a \emph{family of mean functions} if the following conditions hold:
\begin{enumerate}
        \item For each $x,y\in\cX$, $\theta_{xy}\in C( [0,\infty)\times [0,\infty);[0,\infty))$ is positively $1$-homogeneous, jointly concave, nondecreasing in each variable and normalized (that is $\theta(1,1)=1$).
        \item For each $x,y\in\cX$ and $r,s\geq 0$, $\theta_{xy}(r,s)=\theta_{yx}(s,r)$.
    \end{enumerate}
\end{defi}

\begin{example}
    Examples of mean functions are the arithmetic mean $\theta(r,s)=\frac{r+s}2$, the geometric mean $\theta(r,s)=\sqrt{rs}$, the harmonic mean $\theta(r,s)=\frac{2rs}{r+s}$, the logarithmic mean $\theta(r,s)=\frac{r-s}{\log r-\log s}$, or the minimum $\theta(r,s)=\min\{r,s\}$. These are all symmetric means. All of these means can be reweighted, e.g. $\theta(r,s)=\lambda r+(1-\lambda) s$ for $\lambda\in[0,1]$ for the arithmetic mean.
\end{example}

It will be crucial to work from time to time with the localised energy, that is for $A\subset \R^n$ Borel we define
\begin{align*}
    F( m,J,A)=&\sum_{(x,y)\in \cE}\sigma(x,y)\frac{\Hm^1([x,y]\cap A)}{|x-y|} \frac{|x-y|^2| J(x,y)|^2}{\theta_{xy}(m(x),m(y))}.
\end{align*}
Note that $F(m,J,\cdot)$ is $\sigma$-additive, i.e.\ $F(m,J,\bigcup_{i\in\N } A_i)= \sum_{i=1}^\infty F(m,J,A_i)$ whenever $(A_i)_{i\in\N}$ are pairwise disjoint Borel sets. For $A=\R^n$ we have $F(m,J,\R^n)=F(m,J)$.

\begin{defi}\label{defi: action}
    For $m\in W^{1,1}( (0,T); \calP(\cX))$ and $J\in L^1((0,T); \R_a^{\cE})$ we define
   the action functional
\begin{align*}
    \calA(m,J)=\int_0^T F( m_t,J_t)\, dt
\end{align*}
whenever 
\begin{align}\label{eq: dcont}
    \partial_tm_t(x)+\sum_{y\sim x}J_t(x,y)=0 \quad\text{ for all }x\in\cX,
\end{align}
and infinity otherwise.
\end{defi}

The notion of this action goes back to the independent works \cite{maas2011gradient} and \cite{mielke2011gradient} and has been considered a lot since then, for example in \cite{erbar2012ricci,mielke2013geodesic,erbar2012gradient,erbar2016gradient,garcia2020gromov,gigli2013gromov, erbar2020computation}.

We note that the concavity assumption of each $\theta_{xy}$ yields that 
$(y, s, t) \mapsto \frac{y^2}{\theta_{xy}(s,t)}$ is convex. Together with the linearity of \eqref{eq: dcont} this implies the convexity of $(m,J)\mapsto \calA(m,J)$. For more details, we refer the reader to \cite{erbar2012ricci}, in particular Lemma 2.7 and Corollary 2.8.

\subsection{Continuous embedding}
We embed the probability measures $m\in \calP(\cX)$ into the Borel probability measures on $\R^n$ in the natural way. We also embed the flows $J\in\R^{\cE}_a$ into $\M(\R^n;\R^n)$, i.e. the $\R^n$-valued Radon measures , in the following way:
\begin{defi}\label{defi: iota}
    For $J\in \R^{\cE}_a$ we define $\iota J\in \M(\R^n;\R^n)$ by 
    \begin{align*}
        \iota J=\frac12\sum_{(x,y)\in\cE}J(x,y)\frac{y-x}{|y-x|}\Hm^1|_{[x,y]}. 
    \end{align*}
\end{defi}

Note that we count every edge twice, once as $(x,y)$ and once as $(y,x)$. 
The embedding of a solution to the discrete continuity equation solves the continuity equation:
\begin{lemma}\label{lemma: cont}
    Let $m\in W^{1,1}((0,T);\calP(\cX))$ and $J\in L^1((0,T);\R^{\cE}_a)$ solve \eqref{eq: dcont}. 
    Then $(m,\iota J)$ solves 
    \begin{align}\label{eq: ccont}
    \partial_tm+\dive \iota J=0 \quad \text{ in }\mathcal{D}'((0,T)\times \R^n).
\end{align}
\end{lemma}

\begin{proof}
    Let $\varphi\in C_c^\infty((0,T)\times \R^n)$. Then by definition
    \begin{align*}
        \langle \partial_t \varphi, m \rangle
=       -\int_0^T \sum_{x\in\cX} \varphi(t,x) \del_t m_t(x)\, dt = \int_0^T \sum_{x\in\cX} \sum_{y\sim x} J_t(x,y)\varphi(t,x)\, dt,
    \end{align*}
    as well as
    \begin{align*}
     \langle \nabla \varphi, \iota J \rangle
=  \int_0^T \frac12\sum_{(x,y)\in\cE}\int_{[x,y]}\nabla \varphi\cdot \frac{y-x}{|y-x|}J_t(x,y)\, d\Hm^1.
    \end{align*}
    The assertion follows then by the fundamental theorem of calculus along the line segment $[x,y]$.
\end{proof}

 We utilize two notions from measure theory:

For a vector-valued Radon measure $\sigma\in \M(\R^n;\R^m)$, we define its \emph{total variation} as the nonnegative Radon measure $|\sigma|\in \M_+(\R^n)$ through
\begin{equation*}
    |\sigma|(A) = \sup\left\{\int_A \phi \cdot d\sigma\,:\,\phi\in C_c(\R^n;\R^m), \|\phi\|_\infty\leq 1\right\}\quad\text{ for }A\subseteq \R^n\text{ Borel.}
\end{equation*}

For a signed Radon measure $\sigma\in \M(\R^n)$ we define the (possibly infinite) \emph{KR-norm} by
$$\|\sigma\|_{\KR}=\sup\langle f, \sigma\rangle,$$    
where the supremum is taken among all $1$-Lipschitz continuous compactly supported functions $f\in C_c(\R^n)$.

\subsection{A-priori bounds}
\begin{lemma}\label{lem: a prioiri}
Suppose \eqref{path} and \eqref{degree} in Assumption \ref{ass: graph} hold. Let $\bar m^\eps_0,\bar m^\eps_1 \in \calP(\cXe)$. Then
\begin{equation}
    c(\lambda,\max_{x\in \cXoe}\deg(x)) W_1^2(\bar m^\eps_0,\bar m^\eps_1) \leq \min_{\substack{m^\eps_0 = \bar m^\eps_0,\\ m^\eps_1 = \bar m^\eps_1}}\calAe(m^\eps,J^\eps)  \leq C(R,\Lambda)(W_2^2(\bar m^\eps_0,\bar m^\eps_1)+\eps^2).
\end{equation}
\end{lemma}

\begin{proof}
We first prove the upper bound: By convexity of $\calAe$, it suffices to prove this for $\bar m^\eps_0=\delta_x$, $\bar m^\eps_1 = \delta_y$. Using Assumption \ref{ass: graph} \eqref{path}, we find a path $P=(x_0,\ldots,x_N)$ in $(\cXe,\cEe)$ with $x_0 = x$, $x_N = y$ of Euclidean length $L(P)\leq R(|x-y|+\eps)$. We construct the curve $(m^\eps, J^\eps)$ by first setting $m^\eps_{t_k} = \delta_{x_k}$, $t_k = \frac{L(x_0,\ldots,x_k)}{L(P)}$, and then using the original two-point construction from \cite{maas2011gradient} in between $t_k$ and $t_{k+1}$. We can calculate explicitly
    \begin{equation}
        \calAe(m^\eps,J^\eps)= \sum_{k=0}^{N-1} \frac{|x_{k+1}-x_k|^2}{t_{k+1}-t_k}\sigma(x_k,x_{k+1}) C(\theta_{x_k,x_{k+1}}) \leq 8\Lambda L(P)^2 ,
    \end{equation}
where we used \cite[Theorem 2.4]{maas2011gradient}, which states that the diameter of the two-point space for a given mean $\theta$ is
\begin{equation}
    C(\theta) = \left(\int_0^1 \frac{ds}{\sqrt{\theta(s,1-s)}}\right)^2 \leq \left(\int_0^1 \frac{ds}{\min(s,1-s)}\right)^2 = 8. 
\end{equation}
The existence of minimizing curves follows from this by the direct method of the calculus of variations.

To prove the lower bound, 
we estimate the $1$-Wasserstein distance between $\bar m^\eps_0$ and $\bar m^\eps_1$: Let $(m^\eps,J^\eps)$ be a curve that satisfies the discrete continuity equation. Then by the Kantorovich-Rubinstein duality \cite[Remark 7.5]{villani2021topics} and the Cauchy-Schwarz inequality
 \begin{equation} \begin{aligned}\label{eq: a priori}
      W_1^2(\bar m_0^\eps,\bar m_1^\eps) =& \|\bar m_1^\eps -\bar m_0^\eps\|_\KR^2\\
      =& \sup\left\{\int_{(0,1)\times \R^n} \phi \,d(\dive \iota J^\eps)\,:\,\|\nabla \phi\|_\infty \leq 1 \right\}^2\\
      =&\sup\left\{\int_{(0,1)\times \R^n} -\nabla \phi \cdot\,d( \iota J^\eps)\,:\,\|\nabla \phi\|_\infty \leq 1 \right\}^2\\
      \leq& (|\iota J^\eps|((0,1)\times \R^n))^2\\
      \leq &\int_0^1\lambda^{-1} \sum_{(x,y)\in\cEe}\theta_{xy}(m^\eps_t(x),m^\eps_t(y))\, dt\, \calAe(m^\eps,J^\eps)\\
      \leq &2\lambda^{-1}\max_{x\in\cXe}\deg(x) \calAe(m^\eps,J^\eps).
  \end{aligned}
  \end{equation}
  Here, we estimate $\theta_{xy}(r,s)\leq r+s$ and note that each $m_t^\eps$ is a probability measure, which implies
  \begin{align*}
\sum_{(x,y)\in\cEe}\theta_{xy}(m^\eps_t(x),m^\eps_t(y))\leq 2\max_{x\in\cXe}\deg(x).
  \end{align*}
Since $(m^\eps,J^\eps)$ is arbitrary, the assertion follows.
\end{proof}

\subsection{Scaling law}\label{sec: counterexamples}

As long as the vertex degree is bounded, Lemma \ref{lem: a prioiri} suggests a finite nonzero limit action.

For a situation with very high vertex degree, consider the graph $\cXe = \eps \Z$, with every vertex connected to all other vertices of distance at most $N\eps$, $N\in\N$. The vertex degree is $2N$ everywhere.

For a typical situation, consider the unit density $m(x)=\eps$ and total unit flux through every point, which can be achieved by setting $J(x,y)= \frac{2}{N(N+1)}\approx N^{-2}$ whenever $x\leq y\leq x+N\eps$, since the number of edges crossing a generic point is $\frac{N(N+1)}{2}$.

In this situation the energy in $[0,1]$ is of order
\begin{align*}
    F(m,J;[0,1]) \approx \frac1\eps \sum_{k=1}^N (\eps k)^2 \frac{N^{-4}}{\eps}\approx N^{-1}.
\end{align*}

If $N$ does not depend on $\eps$, the maximum degree is constant with $\eps$ and we expect a finite nonzero limit action by Lemma \ref{lem: a prioiri}.

By contrast, in \cite{garcia2020gromov}, $\lim_{\eps \to 0} N_\eps = \infty$, and the article's main result \cite[Theorem 1.22]{garcia2020gromov} shows $\Gamma$-convergence (in more general situations) for the differently-scaled
\begin{align*}
N^n_\eps \calAe \stackrel{\Gamma}{\longrightarrow} C A_{W_2}
\end{align*}
to a specific nonzero multiple of the Euclidean Wasserstein action. In particular, the action as we define it here tends to zero.

These two results suggest the ``correct" energy  for unbounded vertex degrees to scale linearly with vertex degree, e.g.
\begin{align*}
G(m,J) = \sum_{(x,y)\in \cE}\sigma(x,y)\frac{|x-y|^2| J(x,y)|^2}{\theta_{xy}\left(\frac{m(x)}{\deg(x)},\frac{m(y))}{\deg(y)}\right)}.
\end{align*}

However, without Assumption \ref{ass: graph} \eqref{degree}, in general the limit of $G$ may be infinite while the limit of $F$ is finite, as the following example shows:

\begin{example}
Start with the line graph $\eps \Z$ with nearest neighbors. Attach to every base vertex $z\in \eps \Z$ a distinct rescaled copy of the complete graph $z+\eps K_N$, so that every base vertex is connected to its two base neighbors and its entire cul-de-sac. The vertex degrees are either $N$ for attached points or $N+2$ for base points.

Since moving mass macroscopically is only possible between base points $z\in \eps \Z$, we must have $|J(z,z+\eps)|\approx 1$ at most base points at most times. However, the average mass per base point in $[0,1]$ is at most $m(z)=\frac{\text{total mass}}{\text{number of base points}}\approx \eps$, so that
\begin{align*}
F_{\eps,N} \approx 1,\quad G_{\eps,N}\approx N.
\end{align*}

\end{example}

The above two examples show that in the case of unbounded vertex degree, the scaling of the action depends not on the number of neighbors but on the number of useful neighbors. Finding a general scaling law for unbounded vertex degrees appears to be an interesting problem which is not addressed in this article.

\subsection{Finite action curves}

In this section, we define finite action curves, which are mass-flux pairs with finite $W_2$ action. See Figure \ref{fig: FA} for an illustration. As a reference see e.g. \cite{ag} for a comprehensive study of the $W_2$-distance and \cite{ags,santambrogio2015optimal,villani2021topics} for the general $W_p$-case.

\begin{defi}\label{defi: finite action}
    Let $(\rho,j)\in \M((0,T)\times \R^n;[0,\infty)\times \R^n)$ be a finite Radon measure. We say that $(\rho,j)$ is a \emph{finite action curve} if
    \begin{enumerate}
    \item $(\rho,j)$ solve the continuity equation 
    \begin{align*}
\del_t\rho+\dive j=0 \text{ in }\mathcal D'((0,T)\times \R^n).
\end{align*}
        \item $j\ll \rho$ and 
        \begin{align*}
            \int_{(0,T)\times \R^n}\left|\frac{dj}{d\rho}\right|^2\, d\rho<\infty;
        \end{align*}
    \end{enumerate}
\end{defi}

\begin{figure}
\centering
\includegraphics{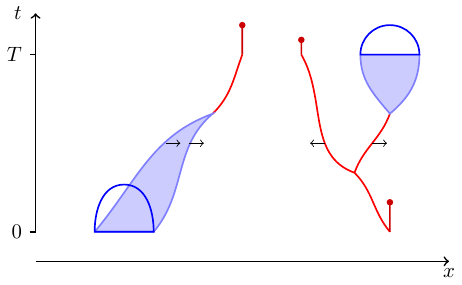}
\caption{A finite action curve $(\rho, j)$. The probability measure $\rho$ has singular (red) and absolutely continuous (blue) parts, with distributions at initial and finite times $\rho_0$, $\rho_T$. The arrows represent the velocities $d j/d \rho$.}\label{fig: FA}
\end{figure}
\begin{rem}
    If $(\rho,j)$ solve the continuity equation we can disintegrate $\rho$ in time, i.e. $\rho=dt\otimes \rho_t$, where $t\mapsto \rho_t$ is a curve of finite Radon measures with constant mass. If additionally $j\ll\rho$, $j$ can also be disintegrated, $j=dt\otimes j_t$ and $\rho_t$ becomes an absolutely continuous curve in the Wasserstein space $W_1$
    \begin{align*}
        W_1(\rho_t,\rho_s)\leq \int_s^t|j_r|(\R^n)\, dr,
    \end{align*}
    by virtue of the Kantorovich-Rubinstein duality \cite[Remark 7.5]{villani2021topics}.
    Since we assume square-integrability of $\frac{dj}{d\rho}$ we even get H\"older continuity \cite[Theorem 2.29]{ag}:
    \begin{align*}
        W_2^2(\rho_t,\rho_s)\leq (t-s)\int_s^t\int_{\R^n} \left|\frac{dj_r}{d\rho_r}\right|^2\, d\rho_r\, dr .
    \end{align*}
    This is an instance of the Benamou-Brenier formula \cite[Proposition 2.30]{ag}.

\end{rem}

\subsection{Compactness}\label{sec: compactness}
\begin{prop}\label{prop: compactness}

Suppose \eqref{edge} and \eqref{degree} in Assumption \ref{ass: graph} hold.    Let \\$m^\eps\in W^{1,1}((0,T);\calP(\cXe))$ and $J^\eps\in L^1((0,T);\R^{\cEe}_a)$ for every $\eps>0$ with
     \begin{align}\label{eq: bounded}
    \sup_{\eps>0}\calAe(m^\eps,J^\eps)<\infty
     \end{align}
     and $m^\eps_0\weakstar \rho_0$ as $\eps\to 0$. Then there exists a subsequence such that $m^\eps\weakstar\rho$ in $\M_+((0,T)\times \R^n)$ and $\iota J^\eps\weakstar j$ in $\M((0,T)\times \R^n;\R^n)$, where
     $(\rho,j)$ is a finite action curve.
     In addition 
     \begin{align*}
        m_t^\eps\weakstar\rho_t 
     \end{align*} for every $t\in[0,T]$ as $\eps\to0$.
\end{prop}

\begin{proof}
We start by proving the narrow convergence of $m^\eps$.
 For this we use the lower bound in the a-priori estimate Lemma \ref{lem: a prioiri} together with a change of variables to get 
 \begin{equation} \begin{aligned}\label{eq: a priori1}
      W_1(m_t^\eps,m_s^\eps)
      \leq &C\sqrt{|t-s|}\sqrt{\calAe(m^\eps,J^\eps)}.
  \end{aligned}
  \end{equation}
  Then, since $m^\eps_0$ is narrowly converging, the set $\{m^\eps_0\}$ is tight by Prokhorov's theorem. Together with \eqref{eq: a priori1}, this implies that the set $\{m_t^\eps: \eps>0, t\in[0,T]\}$ is contained in a $W_1$-ball of a tight set. Since $\R^n$ is locally compact, the $W_1$-ball of a tight set is tight, which means $\{m_t^\eps: \eps>0, t\in[0,T]\}$ is tight itself. 
 By \eqref{eq: a priori1} and tightness the Arzel\`a-Ascoli theorem \cite[Proposition 3.3.1]{ags} implies narrow convergence for every $t\in[0,T]$.

We infer by dominated convergence that $m^\eps$ narrowly converges to $\rho\in \M_+((0,T)\times \R^n)$. In order to show tightness of $\{\iota J^\eps\}$ in $\M((0,T)\times \R^n;\R^n)$, replace $m_t^\eps$ by $\rho_t^\eps$ defined by
  \begin{align*}
      \rho_t^\eps=\sum_{x\in\cXe}\frac1{\deg x}\sum_{y\sim x}U_{[x,y]}m_t^\eps(x),
  \end{align*}
  where $U_{[x,y]}=\frac{\Hm^1|_{[x,y]}}{|y-x|}$ denotes the uniform distribution on the line segment $[x,y]$. By Assumption \ref{ass: graph} \eqref{edge}, $\rho^\eps\weakstar \rho$ in $\M_+((0,T)\times \R^n)$. Then, since on each line segment $[x,y]$
  \begin{align*}
      \frac{d\rho^\eps_t}{d\Hm^1}\geq \frac1{\deg x}\frac{m^\eps_t(x)}{|y-x|}+\frac1{\deg y}\frac{m^\eps_t(y)}{|y-x|},
  \end{align*}
  we have
  \begin{align*}
      \int_{[x,y]}\left|\frac{d\iota J^\eps_t}{d\rho_t^\eps}\right|^2\, d\rho_t^\eps\leq &\int_{[x,y]}\left|\frac{d\iota J^\eps_t}{d\Hm^1}\right|^2\left(\frac{d\rho^\eps_t}{d\Hm^1}\right)^{-1}\, d\Hm^1\\
      \leq &C\int_{[x,y]}\frac{J^\eps_t(x,y)^2|y-x|}{m^\eps_t(x)+m^\eps_t(y)}\, d\Hm^1\\
      \leq &C\frac{J^\eps_t(x,y)^2|y-x|^2}{\theta_{xy}(m^\eps_t(x),m^\eps_t(y))}
  \end{align*}
  and consequently
  \begin{align*}
       \int_0^T\int_{\R^n}\left|\frac{d\iota J^\eps_t}{d\rho_t^\eps}\right|^2\, d\rho_t^\eps\, dt\leq C\calAe(m^\eps,J^\eps).
  \end{align*}
  Together with the tightness of $\{\rho^\eps\}$ we infer that $\{\iota J^\eps\}$ is tight in $\M((0,T)\times \R^n;\R^n)$: Let $I\times A\subset (0,T)\times \R^n$ be Borel, then similarly as in \eqref{eq: a priori}, using additionally Assumption \ref{ass: graph} \eqref{edge}, we bound the total variation
  \begin{align*}
      |\iota J^\eps_t|(I\times A)\leq C\calAe(m^\eps,J^\eps)^{1/2}\rho^\eps(I\times B(A,R\eps))^{1/2}.
  \end{align*}
  Hence $\{\iota J^\eps\}$ is tight and converges (up to a subsequence) to some $j\in \M((0,T)\times \R^n;\R^n)$. Since $(m^\eps,\iota J^\eps)$ satisfies the continuity equation \eqref{eq: ccont} by Lemma \ref{lemma: cont}, it follows that $(\rho,j)$ satisfies it as well as \eqref{eq: ccont} is stable under narrow convergence. We conclude the argument by using the lower semicontinuity result from \cite[Theorem 2.34]{ambrosio2000functions} to get $j\ll\rho$ as well as
  \begin{align}\label{eq: AFP}
      \int\left|\frac{dj}{d\rho}\right|^2\, d\rho\leq \liminf_{\eps\to0}\int\left|\frac{d\iota J^\eps}{d\rho^\eps}\right|^2\, d\rho^\eps\leq C\liminf_{\eps\to0}\calAe(m^\eps,J^\eps).
  \end{align}
  This proves that $(\rho,j)$ is a finite action curve.

\end{proof}

\begin{rem}
     Note that \eqref{eq: AFP} is a nonoptimal version of the lower bound in Theorem \ref{thm: main}. In the case $(\cX,\cE)=(\Z^n,n.n.)$ with constant $\theta$, the constant in \eqref{eq: AFP} can be chosen as $1$, which is optimal, cf. \cite{gigli2013gromov}.
    \end{rem}

\subsection{Cell formula}\label{sec: rescaled graph}
For every $\eps>0$ we consider the rescaled graph $(\cXe,\cEe) = (\eps \cX , \eps \cE)$.

\begin{assumption}
\label{ass: cell}
We assume the following homogenized energy density exists and is independent of the choice of orthotope $Q\subset \R^n$:
    \begin{align}\label{eq: cell}
    f(v)=\lim_{\eps\to0}\inf_{(m^\eps,J^\eps)} \frac{F_\eps(m^\eps,J^\eps,Q)}{\Lm^n(Q)},
\end{align}
where $v\in\R^n$ and the infimum is taken over all pairs $m^\eps\in \calM_+(\cXe)$, $J^\eps\in \R^{\cEe}_a$ such that
\begin{equation}\label{eq: rep}
\begin{aligned}
    &\sum_{x\in\cXe}m^\eps(x)=\Lm^n(Q)\\
    &J^\eps(x,y)=\mathcal{J}^\eps_v(x,y)\quad \forall (x,y)\in \cEe \text{ with } \dist([x,y],\R^n\setminus Q) \leq R\eps\\
    &\sum_{y\in\cXe: y \sim x}J^\eps(x,y)=0 \quad\forall x\in\cXe.
\end{aligned}
\end{equation}
Here, $\mathcal J^\eps_v\colon \cEe\to\R$ is a uniform representative of the vector $v\in \R^n$, see Definition \ref{def: uniform}.
\end{assumption}

While this assumptions may look hard to satisfy, it is satisfied almost surely for stationary random graphs by Kingman's subadditive ergodic theorem, see Proposition \ref{prop: ergodic} below.

\subsection{Random graphs and ergodic theorem}

In the following let $(\cXo,\cEo)_{\omega\in\Omega}$ be a random graph in $\R^n$, where $(\Omega,\calA,\mathbb{P})$ is a probability space. In addition, let $(\sigma_\omega \in [\lambda,\Lambda]^{\cEo})_{\omega\in \Omega}$ be random weights, with $0<\lambda<\Lambda$, and $(\theta_\omega)_{\omega\in\Omega}$ be random mean functions. We now state our definition of stationary and ergodic random weighted graphs, see e.g. \cite{walters2000introduction}.

\begin{defi}\label{defi: stationary}
We call $(\cXo,\cEo,\sigma_\omega,\theta_\omega)_{\omega\in\Omega}$ \emph{stationary} if for every $z\in\Z^n$ the law of the weighted graph $(\cXo,\cEo,\sigma_\omega,\theta_\omega)_{\omega\in\Omega}$ is the same as the law of the translated weighted graph $(\cXo+z,\cEo+z,\sigma_\omega(\cdot - z),\theta_\omega(\cdot - z))_{\omega\in\Omega}$. We call $(\cXo,\cEo,\sigma_\omega,\theta_\omega)_{\omega\in\Omega}$ \emph{ergodic} if all events $A\in \calA$ that are invariant under translations in $\Z^n$ have probability $\mathbb{P}(A)\in\{0,1\}$.
\end{defi}

Note that stationarity does not mean that the random graph itself is translationally invariant.

\begin{prop}\label{prop: ergodic}
    Let $(\cXo,\cEo,\sigma_\omega,\theta_\omega)_{\omega\in\Omega}$ be stationary, with $(\cXo,\cEo)$ satisfying Assumption \ref{ass: graph} \eqref{path}, \eqref{ball} almost surely. Let $(\cXoe,\cEoe)$ be the rescaled random graph with the same weights and means. Then almost surely the limit  in the cell formula \eqref{eq: cell} exists, is finite, and is independent of the orthotope. If in addition the graph, the weights, and the mean functions are ergodic, then the limit in the cell formula is deterministic almost surely.
\end{prop}
\begin{proof}
    This follows from Kingman's subadditive ergodic theorem \cite[Theorem 4.1]{licht2002global}. Here we view $Q\mapsto \inf_{(m,J)}F(m,J,Q)$, where the infimum is taken among solutions of \eqref{eq: rep}, as the subadditive stationary set function. Assumption \ref{ass: graph} \eqref{path},\eqref{ball} are necessary to define the uniform representative $\calJ^\eps_v$ which appears in the cell formula. We find by a discrete change of variables
    \begin{equation}
    \lim_{\eps \to 0} \inf_{(m^\eps,J^\eps)} \frac{F_\eps(m^\eps,J^\eps,Q)}{\Lm^n(Q)} = \lim_{\eps\to 0} \inf_{(m,J)} \frac{\eps^n F(m,J,Q/\eps)}{\Lm^n(Q)} = \lim_{N\to \infty} \inf_{(m,J)}\frac{F(m,J,NQ)}{\Lm^n(NQ)},
    \end{equation}
    where the existence and independence of $Q$ of the final limit is precisely the statement of Kingman's subadditive ergodic theorem.
    \end{proof}

Here are two examples of stationary random graphs satisfying Assumption \ref{ass: graph}:

\begin{example}
    A simple example is the \emph{random conductance model} (see Figure \ref{fig:random}A): The deterministic graph $(\cXo,\cEo)=(\Z^n,n.n.)$ is given by the lattice $\Z^n$ with nearest neighbor interaction. The random weights $(\sigma_\omega\colon n.n.\to[\lambda,\Lambda])_{\omega\in\Omega}$ are the reciprocals of the random conductances. If the weights are stationary, e.g. if they are i.i.d., Assumption \ref{ass: cell} is satisfied. Random walks in random weighted graphs have been considered in e.g. \cite{biskup,andres2013invariance}.
\end{example}

\begin{figure}[h!]
  \centering
  \begin{subfigure}[b]{0.4\linewidth}
    \includegraphics[width=\linewidth]{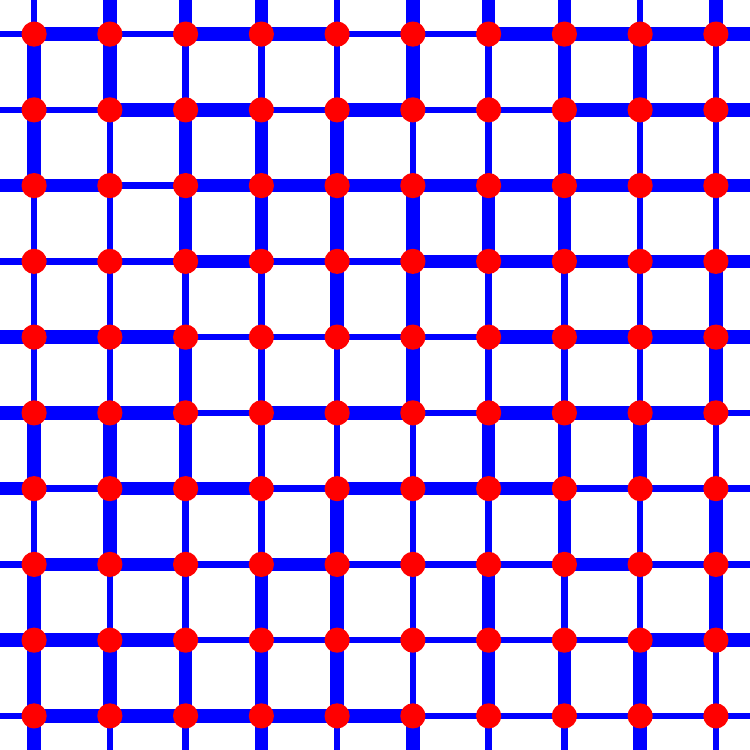}
    \caption{}
  \end{subfigure}
  \hspace{1cm}
  \begin{subfigure}[b]{0.4\linewidth}
    \includegraphics[width=\linewidth]{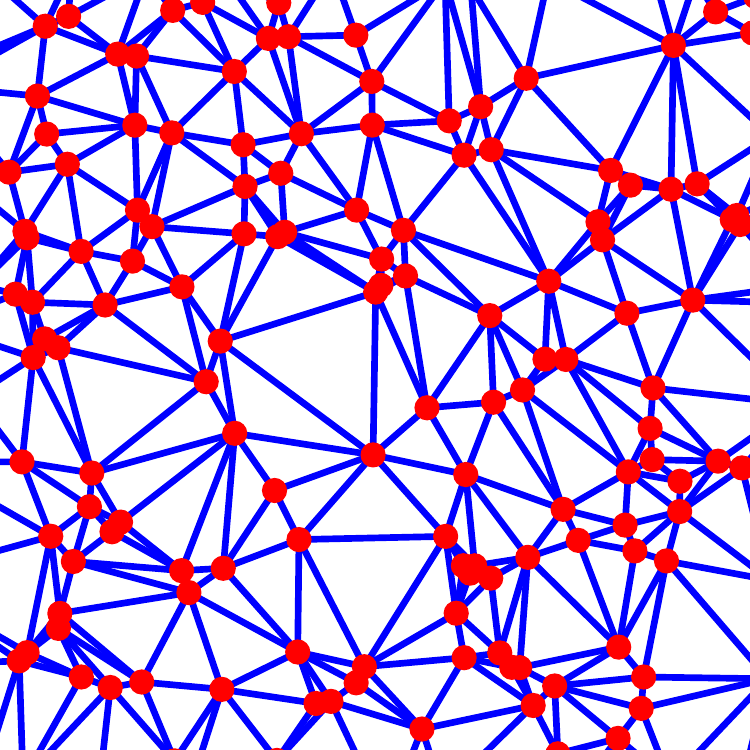}
    \caption{}
  \end{subfigure}
  \caption{Two stationary random graphs: (A) Random conductance model, (B) Dual Voronoi (Delaunay) triangulation of random points.}
  \label{fig:random}
\end{figure}

\begin{example}
    The Poisson Point Process does not satisfy Assumption \ref{ass: graph}\eqref{ball} almost surely. By contrast, the following point process (see Figure \ref{fig:random}B) satisfies Assumption \ref{ass: graph}: Define $\cXo= \{z+\xi_\omega(z)\,:\,z\in \Z^n\}$, where $(\xi(z))_{z\in \Z^n}$ are i.i.d. bounded random shifts. Define the edges through the dual Voronoi tessellation: Two vertices $z+\xi_\omega(z)$, $z'+\xi_\omega(z')$ are connected if and only if their Voronoi cells $A_\omega(z)$, $A_\omega(z')$ have an interface in the sense that $\Hm^{n-1}(\partial A_\omega(z) \cap \partial A_\omega(z'))>0$. The resulting graph satisfies Assumption \ref{ass: graph} by the boundedness of the shifts:
    \eqref{ball} holds for $R=\|\xi\|_\infty+\sqrt{n}/2$. \eqref{edge} holds for $R=2\|\xi\|_\infty+ \sqrt{n}$. \eqref{degree} holds, since the vertex degree is bounded by $O((R+1)^n)$. \eqref{path} holds because the number of points on the shortest path between $z+\xi_\omega(z)$ and $z+e_i+\xi_\omega(z+e_i)$ is bounded by $O((R+1)^n)$.
\end{example}

Since the rest of the article uses no probability theory, we will only work with Assumption \ref{ass: cell} instead of stationarity.

\subsection{Boundary values in the cell formula}

The cell formula \eqref{eq: rep} allows only flows $J^\eps$ that satisfy the boundary values $J^\eps = \calJ_v^\eps$ near $\partial Q$. In continuous cell formulas, the canonical choice of boundary value in this context is the constant $j^\eps = v$ on $\partial Q$. In the graph setting, there is no uniform flow of magnitude $v$. However, for graphs satisfying \eqref{path} and \eqref{ball} in Assumption \ref{ass: graph}, there are flows $\calJ_v^\eps$ that are approximately uniform, which we call \emph{uniform representatives}. We give the precise definition and show their existence in this section.

The following lemma provides a mapping $\phi$ from $\Z^n$ into $(\cX,\cE)$, sending nearest neighbor pairs in $\Z^n$ to paths in $(\cX,\cE)$.

\begin{lemma}\label{lemma: homo}[{Copy of $\Z^n$ in $(\cX,\cE)$}]
Assume \eqref{path} and \eqref{ball} in Assumption \ref{ass: graph}. Then there is a graph homomorphism $\phi$ from a subdivision of $ (\Z^n,n.n.)$ to $(\cX,\cE)$, i.e.\ a map $\phi: \Z^n\to\cX$ and a map also called $\phi: n.n.\to \{\text{paths in }(\cX,\cE)\}$ such that $\phi((z,z'))$ is a path from $\phi(z)$ to $\phi(z')$ , with  
\begin{align*}
    &|\phi(z)-z|\leq R, \quad\forall z\in\Z^n\\
    &L(\phi(z,z'))\leq 2R(R+1) \quad \forall z,z'\ n.n. \text{ in } \Z^n.
\end{align*}
\end{lemma}
\begin{proof}
    For every $z\in\Z^n$ we choose $\phi(z)=x\in \cX\cap B(z,R)$ according to assumption \eqref{ball}. Then the first assertion follows. To define $\phi(z,z')$, choose a path $P$ from $\phi(z)$ to $\phi(z')$ of length $L(P)\leq R(|\phi(z)-\phi(z')|+1)\leq R(2R+2)$. 
\end{proof}

By rescaling, there is also a copy of $\eps \Z^n$ in $(\cXe,\cEe)$.

\begin{defi}\label{def: uniform}
    A uniform representative of a vector $v\in \R^n$ is a flow $\calJ^\eps_v \in \R^{\cEe}_a$ defined through
    \[
    \calJ^\eps_v = \phi_\eps^\# v_{\eps},
    \]
    where $\phi_\eps:(\eps \Z^n, n.n.) \to (\cXe,\cEe)$ is as in Lemma \ref{lemma: homo} (but rescaled) and $v_{\eps}\in \R^{n.n.}_a$,$v_\eps(z,z\pm \eps e_i) = \pm \eps^{n-1} v_i$ is the uniform flow in direction $v$ on $(\eps \Z^n,n.n.)$. 
\end{defi}

In the above definition, the pushforward of a flow $v\in \R^\cE_a$ under a graph homomorphism $\phi$ from a subdivision of $(\cX,\cE)$ into $(\cX',\cE')$ is defined as
\begin{equation}
    \phi^\#v = \frac12 \sum_{(x,y)\in \cE} v(x,y) J_{\phi(x,y)},
\end{equation}
where $\phi(x,y) = (\phi(x),r_1,\ldots,r_{N-1},\phi(y))$ is a path in $(\cX',\cE')$, and for any path $P= (r_0,\ldots,r_N)$ in $(\cX',\cE')$ the unit flow along $P$ is defined as $J_P\in \R^{\cE'}_a$,
\begin{equation}
J_P =\sum_{k=0}^{N-1} (\delta_{(r_k,r_{k+1})} - \delta_{(r_{k+1},r_k)}).
\end{equation}

The following lemma is taken directly from \cite[Proposition 5.4]{gladbach2024stochastic}. We give the short proof for the reader's convenience.
\begin{lemma}[Properties of $\calJ^\eps_v$]\label{lem: prop}
Assume \eqref{path} and \eqref{ball} in Assumption \ref{ass: graph}.
The following assertions hold:

\begin{enumerate}
    \item $\sum_{y\sim x}\mathcal J^\eps_v(x,y)=0$ $\forall x\in\cXe$;
        \item $|\iota\calJ^\eps_v|(A)\leq C |v| \Lm^n(A)$ for every orthotope (hyperbox) $A\subset \R^n$ of side length at least $\eps$ in all directions, where $C=C(R,n)$;
    \item $\iota \calJ^\eps_v\weakstar v\Lm^n$ vaguely.
 \end{enumerate}
    \end{lemma}
    \begin{proof}
        For (1): Clearly, $v_\eps$ is divergence-free in $(\eps \Z^n, n.n.)$. By the above definition, the divergence of $J_{\phi_\eps(z,z')}$ is $\delta_{\phi_\eps(z')} - \delta_{\phi_\eps(z)}$, and summing over all nearest neighbors yields
        \begin{align*}
            \sum_{y\sim x}\mathcal J^\eps_v(x,y)=\sum_{z\in \phi_\eps^{-1}(x)}\sum_{z'\sim z} v_\eps(z,z') =0\quad \text{ for all }x\in\cXoe.
        \end{align*}

For (2): Observe that 
\begin{align*}
    |\iota\calJ^\eps_v|(A)\leq &\sum_{z,z'\, n.n.}|v_\eps(z,z')|\Hm^1(A \cap \phi_\eps(z,z'))\\
    \leq &  \sum_{z\in\eps\Z^n\,:\,\dist(z,A)\leq 2R(R+1)\eps} \eps^{n-1}|v|\Hm^1(\phi_\eps(z,z'))\\
    \leq &C(R,n)\eps^n|v|\frac{\Lm^n(B(A,C(R)\eps))}{\eps^n} \leq C(R,n)|v|\Lm^n(A),
\end{align*}
where in the last step we used that $A$ has side length at least $\eps$ in all directions.

For (3): It suffices to show the assertion for $v=e_i$. Note that (2) implies compactness in the vague topology by virtue of De La Vall\'ee Poussin's theorem, see e.g. \cite[Corollary 1.60]{ambrosio2000functions}.
  The limit must be $e_i\Lm^n$. To see this, take a test function $\varphi\in C_c(\R^n;\R^n)$, then
       \begin{align}\label{eq: test}
       \langle \iota\calJ^\eps_v, \varphi \rangle = &\sum_{z\in \eps\Z^n}\eps^{n-1} \int_{\phi_\eps(z,z+\eps e_i)} \varphi\cdot \tau_\eps\,d\Hm^{1},
       \end{align}
       where $\tau_\eps$ is the normalized tangent vector field along $\phi_\eps(z,z')$. 
We approximate $\varphi$ on $\phi_\eps(z,z')$ by $\varphi(z)$ using the modulus of continuity $\rho: [0,\infty)\to[0,\infty)$ of $\varphi$
 \begin{align*}
       \langle \iota\calJ^\eps_v, \varphi \rangle = &\sum_{z\in \eps\Z^n}\eps^{n-1} \varphi(z)\cdot\int_{\phi_\eps(z,z+\eps e_i)} \tau_\eps\,d\Hm^{1}+O(\rho((R+1)\eps)\Lm^n(B(\supp \varphi,R\eps))\\
       =&\sum_{z\in \eps\Z^n}\eps^{n-1} \varphi(z)\cdot(\phi_\eps(z+\eps e_i)-\phi_\eps(z))+O(\rho((R+1)\eps)\Lm^n(B(\supp \varphi,R\eps)),
       \end{align*}
       where we used the fundamental theorem of calculus in the last step. Making again use of the continuity of $\varphi$, we find for $N\geq 1$
       that 
       \begin{align*}
           &\sum_{z\in \eps\Z^n}\eps^{n-1} \varphi(z)\cdot(\phi_\eps(z+\eps e_i)-\phi_\eps(z))\\
           =&\sum_{z\in \eps\Z^n}\eps^{n-1} \frac1N\sum_{k=1}^N\varphi(z+k\eps e_i)\cdot(\phi_\eps(z+\eps e_i)-\phi_\eps(z))+\rho(N\eps)\\
           =&\sum_{z\in \eps\Z^n}\eps^{n-1} \varphi(z)\cdot\frac{\phi_\eps(z+N\eps e_i)-\phi_\eps(z)}{N}+\rho(N\eps).
       \end{align*}
       Since $|\phi_\eps(z)-z|\leq R\eps$ by Lemma \ref{lemma: homo}, we obtain
          \begin{align*}
        &\sum_{z\in \eps\Z^n}\eps^{n-1} \varphi(z)\cdot\frac{\phi_\eps(z+N\eps e_i)-\phi_\eps(z)}{N}+\rho(N\eps)\\
        =&\sum_{z\in \eps\Z^n}\eps^{n-1} \varphi(z)\cdot \eps e_i+O\left(\sum_{z\in \eps\Z^n}\frac{\eps^{n}}N \varphi(z)\right)+\rho(N\eps).
       \end{align*}
     The first term is a Riemann sum and plugging the last equation back into 
     \eqref{eq: test} we see that for $N\approx 1/\sqrt{\eps}$ 
     \begin{align*}
       \lim_{\eps\to0}  \langle \iota\calJ^\eps_v, \varphi \rangle =\langle e_i,\varphi\rangle.
     \end{align*}

    \end{proof}

\subsection{Homogenized action}

For a finite action curve $(\rho,j)$ we define the homogenized action 
\begin{align}\label{eq: hom action}
    A(\rho,j)=\int_0^T\int_{\R^n}f\left(\frac{dj_t}{d\rho_t}\right)\, d\rho_t\, dt.
\end{align}

Here, the homogenized energy density $f : \R^n \to [0,\infty)$ is given by the cell formula \eqref{eq: cell}.

   \begin{lemma}\label{lemma: f}
Assume \eqref{path}, \eqref{ball} and \eqref{degree} in Assumption \ref{ass: graph}.   For all $v\in\R^n$ there exist constants $c,C>0$ depending only on $R,n,\lambda,\Lambda,$ and the maximum degree such that
   \begin{align*}
       c|v|^2\leq f(v)\leq C|v|^2.
   \end{align*}
   \end{lemma}

   \begin{example}
       Consider the case $(\cX, \cE)=(\Z^n,n.n.)$ with constant $\theta$ and $\sigma(x,y)=1$ for all nearest neighbors. In this case the uniform representative on $\eps\Z^n$ is
       \begin{align*}
           \calJ_v^\eps(\eps z,\eps z')=\eps^{n-1}v\cdot(z'-z)=v_\eps.
       \end{align*}
       Our chosen competitor to the cell formula for $Q=[0,1)^n$ with $\eps=1/N$ is the pair $J^\eps=\calJ_v^\eps$, $m^\eps = \mathds{1}_{\bar{Q}} \frac{\eps^n}{(1+\eps)^n}$. We calculate explicitly
       \begin{align*}
           f(v) \leq \lim_{\eps\to 0} F_\eps(m^\eps,J^\eps,Q) = \lim_{\eps\to 0} 2\sum_{i=1}^n N(N+1)^{n-1} \frac{\eps^2 (\eps^{n-1}v_i)^2}{\theta(\frac{\eps^n}{(1+\eps)^n},\frac{\eps^n}{(1+\eps)^n})} = 2|v|^2.
       \end{align*}
       For the lower bound $f(v)\geq 2|v|^2$, consider any competitor $(m^\eps, J^\eps)$ to the cell problem in $Q=[0,1)^n$. We decompose the energy by direction and use subadditivity:
       \begin{align*}
           F_\eps(m^\eps,J^\eps, Q)= &2\sum_{z\in Q\cap \eps \Z^n}\sum_{i=1}^n\frac{\eps^2 J^\eps(z,z+\eps e_i)^2}{\theta(m^\eps(z),m^\eps(z+\eps e_i))}\\
           \geq& 2\sum_{i=1}^n\frac{(\eps\sum_{z\in Q\cap \eps \Z^n}J^\eps(z,z+\eps e_i))^2}{\theta(\sum_{z\in Q\cap \eps \Z^n}m^\eps(z),\sum_{z\in Q\cap \eps \Z^n}m^\eps(z+\eps e_i))}\\
           =&2\sum_{i=1}^n(\eps\sum_{z\in Q\cap \eps \Z^n}J^\eps(z,z+\eps e_i))^2
           =2\sum_{i=1}^n(\eps\sum_{z\in Q\cap \eps \Z^n}v_\eps(z,z+\eps e_i))^2\\
           =&2\sum_{i=1}^n (\iota J^\eps(Q) \cdot e_i)^2 = 2\sum_{i=1}^n|v_i|^2,
       \end{align*}
       where we used that $J^\eps$ is divergence-free and $\iota J^\eps(Q)$ only depends on its boundary values.
   \end{example}
   \begin{proof}[Proof of Lemma \ref{lemma: f}]

       To prove the upper bound, take the competitor $J^\eps = \calJ^\eps_v$ and
       \[
       m^\eps(x) = \frac{\Lm^n(Q)}{2}\frac{\sum_{y\sim x} |J^\eps(x,y)|\Hm^1(Q\cap [x,y])}{\sum_{(\tilde x, \tilde y)\in \cEe}  |J^\eps(\tilde x,\tilde y)|\Hm^1(Q\cap [\tilde x,\tilde y])},\quad\text{ for all }x\in\cXe.
       \]

       Note that $(m^\eps,J^\eps)$ is a competitor to the cell problem. Due to the positive $1$-homogeneity, normalization, and monotonicity of $\theta_{xy}$ we may estimate 
       \[
       \theta_{xy}(m^\eps(x),m^\eps(y)) \geq \frac{\Lm^n(Q)}{2}\frac{ |J^\eps(x,y)|\Hm^1(Q\cap[x,y])}{\sum_{(\tilde x,\tilde y)\in \cEe}|J^\eps(\tilde x,\tilde y)|\Hm^1(Q\cap [\tilde x,\tilde y])}.
       \]

       Then a direct calculation shows that
       \begin{align*}
           F_\eps(m^\eps,J^\eps,Q)=&\sum_{(x,y)\in \cEe}\sigma(x,y)|x-y|^2\frac{J^\eps(x,y)^2}{\theta_{xy}(m^\eps(x),m^\eps(y))} \frac{\Hm^1([x,y]\cap Q)}{|x-y|}\\
           \leq & \left(\sum_{(x,y)\in \cEe}\! \sigma(x,y)|x-y||J^\eps(x,y)|\mathds 1_{[x,y]\cap Q}\right)\left(\sum_{(\tilde x,\tilde y)\in \cEe}\!|J^\eps(\tilde x,\tilde y)|\Hm^1(Q\cap[\tilde x,\tilde y])\right)\\
           \leq & \frac{C(|J^\eps|(B(Q,R\eps))^2}{\Lm^n(Q)}
           \leq C|v|^2(\Lm^n(Q)+O(\eps)),
       \end{align*}
       where we used Lemma \ref{lem: prop} in the last estimate.
       
       To prove the lower bound, we use that $\theta_{xy}(r,s)\leq r+s$. Then employing the subadditivity of the function $g:\R^n\times [0,\infty) \times [0,\infty)\to[0,\infty]$, $g(j,r,s) = \frac{|j|^2}{r+s}$ we estimate for any competitor $(m^\eps,J^\eps)$ to the cell problem in $Q$
       \[
       \begin{aligned}
F_\eps(m^\eps,J^\eps,Q)\geq       &\lambda\sum_{(x,y)\in \cEe\cap Q\times Q}|x-y|^2\frac{J^\eps(x,y)^2}{\theta_{xy}(m^\eps(x),m^\eps(y))}\\
       \geq &\lambda \frac{\left| \sum_{(x,y)\in \cEe\cap Q\times Q} (y-x)J^\eps(x,y)\right|^2}{\sum_{(x,y)\in \cEe\cap Q\times Q} (m^\eps(x)+ m^\eps(y))}\\
       \geq &\frac{\lambda}{2\max_{x\in\cX}\deg(x) \Lm^n(Q)}|\iota J^\eps(Q)|^2 \geq c|v|^2\Lm^n(Q) - o(1),
        \end{aligned}
       \]
       where we used the fact that $\iota J^\eps$ is divergence-free, implying that $\iota J^\eps(Q) = \iota \calJ_v^\eps(Q) = \Lm^n(Q)v + o(1)$.
      
   \end{proof}

\section{Proof of Corollary \ref{cor: main}}\label{sec: proofc}

 By Lemma \ref{lem: a prioiri}, minimizing curves $(m^\eps,J^\eps)$ exist and satisfy 
 \begin{align*}
     \calAe(m^\eps,J^\eps) \leq CW_2^2(\bar m^\eps_0,\bar m^\eps_1),
 \end{align*}
 which is bounded. By Proposition \ref{prop: compactness}, a convergent subsequence $m^\eps \weakstar \rho$, $\iota J^\eps \weakstar j$ exists, where $(\rho,j)$ is a finite action curve with
\begin{equation}
     C(\rho_0,\rho_1) \leq A(\rho,j) \leq \liminf_{\eps \to 0} \calAe(m^\eps,J^\eps)
\end{equation}
by the lower bound in Theorem \ref{thm: main}.

We now show that 
\begin{equation}\label{eq: upper corollary}
\limsup_{\eps\to 0} \calAe(m^\eps,J^\eps) \leq C(\rho_0,\rho_1).
\end{equation}

By the upper bound in Theorem \ref{thm: main}, there is a sequence of curves $(n^\eps,K^\eps)$ with $W_2(n^\eps_0,\rho_0)\to 0$, $W_2(n^\eps_1, \rho_1)\to 0$, and $\limsup_{\eps \to 0} \calAe(n^\eps, K^\eps) = C(\rho_0,\rho_1)$.

We construct new curves $(o^\eps,M^\eps)$ with $o^\eps_0 = \bar m^\eps_0$, $o^\eps_1 = \bar m^\eps_1$ as follows:

  \begin{itemize}
      \item Let $(p^\eps,N^\eps)$ be the optimal curve connecting $p^\eps_0 = \bar m^\eps_0$ and $p^\eps_1 = n^\eps_0$, with action $\calAe(p^\eps,N^\eps)\leq C (W_2^2(\bar m^\eps_0, n^\eps_0)+\eps^2)\to 0$ by Lemma \ref{lem: a prioiri}.
    \item Let $(q^\eps,O^\eps)$ be the optimal curve connecting $q^\eps_0 = n^\eps_1$ and $q^\eps_1 = \bar m^\eps_1$, with action $\calAe(q^\eps,O^\eps)\leq C (W_2^2(\bar m^\eps_1, n^\eps_1)+\eps^2)\to 0$ by Lemma \ref{lem: a prioiri}.
  \end{itemize}

  Choose $\delta>0$ and construct the global curve $(o^\eps,M^\eps)$ connecting $\bar m^\eps_0$ and $\bar m^\eps_1$ as follows:
  \begin{equation}
o^\eps_t = \begin{cases}
    p^\eps_{t/\delta} &t\in [0,\delta]\\
    n^\eps_{(t-\delta)/(1-2\delta)} &t\in[\delta,1-\delta]\\
    q^\eps_{(t-(1-\delta))/\delta} &t\in[1-\delta,1],
\end{cases}
  \end{equation}
  \begin{equation}
M^\eps_t = \begin{cases}
    \frac{1}{\delta}N^\eps_{t/\delta} &t\in (0,\delta)\\
    \frac{1}{1-2\delta}K^\eps_{(t-\delta)/(1-2\delta)} &t\in(\delta,1-\delta)\\
    \frac{1}{\delta}O^\eps_{(t-(1-\delta))/\delta} &t\in(1-\delta,1).
\end{cases}
  \end{equation}

    By the change of variables formula, $(o^\eps,M^\eps)$ solves the continuity equation \eqref{eq: dcont} and
    \begin{equation}
        \calAe(o^\eps,M^\eps) = \frac{\calAe(p^\eps,N^\eps)}{\delta}+\frac{\calAe(n^\eps,K^\eps)}{1-2\delta} + \frac{\calAe(q^\eps,O^\eps)}{\delta}.
    \end{equation}
Since $(m^\eps, J^\eps)$ are minimizing curves in the same class as $(o^\eps, M^\eps)$,
\begin{equation}
    \limsup_{\eps\to 0} \calAe(m^\eps,J^\eps) \leq \limsup_{\eps\to 0}\calAe(o^\eps,M^\eps) = \limsup_{\eps\to 0}\frac{\calAe(n^\eps,K^\eps)}{1-2\delta} = \frac{C(\rho_0,\rho_1)}{1-2\delta}.
\end{equation}
Since $\delta>0$ is arbitrary, \eqref{eq: upper corollary} follows.

\section{Proof of Theorem \ref{thm: main}}\label{sec: proofs}

\subsection{Proof of the lower bound} \label{sec: lower}
We have to show that for every sequence $m^\eps\subset \R^\cXe_+$, $J^\eps\subset \R^{\cEe}_a$ such that $(m^\eps,\iota J^\eps) \weakstar (\rho,j)$ we have
\begin{align*}
    \liminf_{\eps\to0}\calAe(m^\eps,J^\eps)\geq A(\rho,j).
\end{align*}

We first state a lower bound for the energy at a single time for specific limit measures, which occur naturally as tangent measures after a blow-up procedure. The proposition is an adaptation of \cite[Proposition 10.3]{gladbach2024stochastic} to the quadratic case (note the quadratic error term):

\begin{prop}\label{prop: lower bound tangent measure}
    Let $j_0\in\R^n\setminus \{0\}$, $Q_{j_0,\alpha}$ a closed orthotope with one side parallel to $j_0$ of length one and all other sides of length at most $\alpha$.
    
    Let $(m^\eps,\iota J^\eps,|\iota J^\eps|,\dive \iota J^\eps) \weakstar (1,j_0,c,0)\tau$ narrowly in $Q_{j_0,\alpha}$, where $\tau\in \calP(Q_{j_0,\alpha})$ satisfies $\tau(\partial Q_{j_0,\alpha}) = 0$ and $c>0$. Then
    \begin{equation}
        \liminf_{\eps \to 0} F_\eps(m^\eps,J^\eps,Q_{j_0,\alpha}) \geq f(j_0) - C\sqrt{\alpha}|j_0|^2,
    \end{equation}
    where $C$ depends only on $R,n,\lambda,\Lambda,$ and the maximum degree.
\end{prop}

We postpone the proof for now. Figure \ref{fig: hori} illustrates the orthotope $Q_{j_0,\alpha}$ and the limit flux $j_0\tau$.

With Proposition \ref{prop: lower bound tangent measure} we can deduce the lower bound in Theorem \ref{thm: main}:
\begin{proof}[Proof of the lower bound]
Without loss of generality we may assume
\begin{align*}
    \sup_{\eps>0}\calAe(m^\eps,J^\eps)<\infty.
\end{align*}
First replace $(m^\eps,J^\eps)$ with versions mollified in time, by first extending $m^\eps$ constantly and $J^\eps$ by $0$ outside of $(0,T)$, and then convolving both with a smooth test function in time. We shall not rename $(m^\eps,J^\eps)$, noting that the action is decreased by convolution due to Jensen's inequality and Fubini's theorem by convexity of $\calAe$. We refer the reader to \cite[Lemma 7.7]{gladbach2023homogenisation} for more details, where a similar construction is considered.

By the compactness result Proposition \ref{prop: compactness}, $m^\eps_t\weakstar \rho_t, \iota J^\eps_t\weakstar j_t, \dive(\iota J^\eps_t)\weakstar \dive(j_t)$ narrowly for every(!) $t\in[0,T]$, with $j_t \ll \rho_t$.

We will show for every $t\in[0,T]$ that
\begin{equation}\label{eq: single time lower bound}
    \int_{\R^n} f\left(\frac{dj_t}{d\rho_t}\right)\,d\rho_t \leq \liminf_{\eps \to 0} F_\eps(m^\eps_t,J^\eps_t) .
\end{equation}

The inequality for the actions
\begin{equation}
    \int_0^T\int_{\R^n} f\left(\frac{dj_t}{d\rho_t}\right)\,d\rho_t\,dt \leq \liminf_{\eps \to 0} \calA_\eps(m^\eps,J^\eps) 
\end{equation}
then follows from Fatou's Lemma.

In order to show \eqref{eq: single time lower bound}, we employ the blow-up method: Define the energy density measures $\nu_t^\eps = F_\eps(m^\eps_t,J^\eps_t,\cdot)\in \M_+(\R^n)$. Since the energies are uniformly bounded, we obtain by De La Vall\'ee Poussin's theorem that $\nu_t^\eps\weakstar \nu_t$ vaguely for a subsequence (not relabelled). 

Take any point $x_0\in \supp \rho_t$ with the following properties:
\begin{enumerate}
    \item $\displaystyle{\limsup_{r\to 0} \frac{|j_t|(B(x_0,r))+|\dive j_t|(B(x_0,r))+\nu_t(B(x_0,r))}{\rho_t(B(x_0,r))}<\infty}$
    \item $\displaystyle{\lim_{r\to 0} \frac{j_t(x_0+rA)}{\rho_t(x_0+rA)} =\frac{dj_t}{d\rho_t}(x_0) = j_0\in\R^n}$, $\displaystyle{\lim_{r\to 0} \frac{|\dive j_t|(x_0+rA)}{\rho_t(x_0+rA)} < \infty }$ and \\
    $\displaystyle{\lim_{r\to 0} \frac{\nu_t(x_0+rA)}{\rho_t(x_0+rA)}=\frac{d\nu_t}{d\rho_t}<\infty}$ for all $A\subseteq \R^n$ open, convex, bounded and containing the origin
    \item There is a sequence $r_k\to 0$ and a tangent probability measure $\tau\in \calP(Q_{j_0,\alpha})$ such that
    \begin{equation}
    \frac{\rho_t(x_0+r_k \cdot)}{\rho_t(x_0+r_k Q_{j_0,\alpha})} \weakstar \tau, \quad \frac{j_t(x_0+r_k \cdot)}{\rho_t(x_0+r_kQ_{j_0,\alpha})} \weakstar j_0\tau
    \end{equation}
    narrowly in $Q_{j_0,\alpha}$ and $\tau(\partial Q_{j_0,\alpha})=0$, where $Q_{j_0,\alpha}$ is a closed orthotope as in Proposition \ref{prop: lower bound tangent measure}
    \item $\displaystyle{ \frac{\dive(j_t(x_0+r_k \cdot))}{\rho_t(x_0+r_k Q_{j_0,\alpha})}= r_k \frac{(\dive j_t)(x_0+r_k \cdot))}{\rho_t(x_0+r_k Q_{j_0,\alpha})} \to 0}$ in total variation.
\end{enumerate}

By the refined Besicovitch Differentiation Theorem \cite[Proposition 2.2]{ambrosio1992relaxation} $\rho_t$-almost every $x_0$ satisfies (1)(2). (4) immediately follows from the chain rule and (2). (3) holds for $\rho_t$-almost every $x_0$ by the fundamental property of tangent measures \cite[Theorem 2.44] {ambrosio2000functions} and \cite[Lemma 3.1]{rindler2012lower}.

Now choose a diagonal sequence $\eps_k\to 0$ fast enough that
\begin{enumerate}
    \item[(5)] $s_k=\frac{\eps_k}{r_k}\to 0$
    \item[(6)] $\displaystyle{\frac{m_t^\eps(x_0+r_k \cdot)}{\rho_t(x_0+r_k Q_{j_0,\alpha})} \weakstar \tau}$, $\displaystyle{\frac{\iota J_t^\eps(x_0+r_k \cdot)}{\rho_t(x_0+r_kQ_{j_0,\alpha})} \weakstar j_0\tau}$\quad narrowly in $Q_{j_0,\alpha}$
    \item[(7)] $\displaystyle{\frac{F_{\eps}(m^\eps_t,J^\eps_t, x_0+r_kQ_{j_0,\alpha})}{\rho_t(x_0+r_k Q_{j_0,\alpha})} \to \frac{d\nu_t}{d\rho_t}(x_0)}$
    \item [(8)]  $\displaystyle{ \dive\left(\frac{\iota J^\eps_t(x_0+r_k \cdot)}{\rho_t(x_0+r_k Q_{j_0,\alpha})} \right) \to 0}$ in total variation.
\end{enumerate}

Define on the rescaled and translated graph $(\cX_{s_k},\cE_{s_k})$ the pair $n^{s_k}\in [0,\infty)^{\cX_{s_k}}$, $K^{s_k}\in \R_a^{\cE_{s_k}}$,
\[
n^{s_k}(y) = \frac{m^\eps_t(x_0+r_k y)}{\rho_t(x_0+r_k Q_{j_0,\alpha})}, \quad K^{s_k}(y,y') = \frac{J^\eps_t(x_0+r_k y, x_0+r_ky')}{\rho_t(x_0+r_kQ_{j_0,\alpha})}.
\]

Then $(n^{s_k},\iota K^{s_k},\dive \iota K^{s_k}) \weakstar (1,j_0,0)\tau$ narrowly in $Q_{j_0,\alpha}$, and by Proposition \ref{prop: lower bound tangent measure}
    \begin{equation}
    \begin{aligned}
&f(j_0) - C\alpha|j_0|^2 \leq \liminf_{s_k \to 0} F_{s_k}(n^{s_k},J^{s_k},Q_{j_0,\alpha}) \\
&= \liminf_{k\to \infty} \frac{F_{\eps_k}(m^\eps,J^\eps,x_0+r_k Q_{j_0,\alpha})}{\rho_t(x_0+r_kQ_{j_0,\alpha})} =\frac{d\nu_t}
{d\rho_t}(x_0).  \end{aligned}  \end{equation}

Since $\alpha>0$ was arbitrary, we have
\begin{equation}
    f\left(\frac{dj_t}{d\rho_t}\right) \leq \frac{d\nu_t}{d\rho_t}
\end{equation}
$\rho_t$-almost everywhere. By integration
\begin{equation}
    \int_{\R^n} f\left(\frac{dj_t}{d\rho_t}\right)\,d\rho_t \leq \int_{\R^n} \frac{d\nu_t}{d\rho_t}\,d\rho_t \leq \nu_t(\R^n) \leq \liminf_{\eps\to 0} F_\eps(m^\eps_t,J^\eps_t),
\end{equation}
which is \eqref{eq: single time lower bound}, finishing the proof.
\end{proof}

Before we prove Proposition \ref{prop: lower bound tangent measure}, we state the following lemma, which allows us to interpolate between a horizontally nonconstant tangent measure $j_0\tau\in \M(Q_{j_0,\alpha};\R^n)$ and the constant vector field $j_0$ near $\partial Q_{j_0,\alpha}$ without incurring a lot of divergence, which is possible precisely because $Q_{j_0,\alpha}$ is thin.  The cutoff procedure is adapted from the proof of \cite[Proposition 10.3]{gladbach2024stochastic}. See Figure \ref{fig: hori} for an illustration of the cutoff procedure.

\begin{lemma}[Massaging of the tangent measure]\label{lemma: cutoff}
    Let $\tau,j_0,\alpha$ be as above. Let $U=\frac{\Lm^n}{\Lm^n(Q_{j_0,\alpha})}\in \calM_+(\R^n)$
     be the Lebesgue measure normalized by the volume of $Q_{j_0,\alpha}$.  Then for any $\delta>0$ there is a cutoff function $\eta\in C_c^\infty(Q_{j_0,\alpha};[0,1])$ such that
    \begin{enumerate}
        \item $\eta(x) =1 $ whenever $x\in  Q_{j_0,\alpha}$ with $\dist(x,\partial Q_{j_0,\alpha}) \geq \delta$ 
        \item $\| \dive \left(j_0 \left(\eta \tau + (1-\eta)U\right)\right) \|_{\KR} \leq C|j_0|\alpha$,
    \end{enumerate}
where $C$ depends only on $R$ and $n$.
\end{lemma}

\begin{proof}
    We define $\eta = \eta_\mathrm{hori}\otimes  \eta_\mathrm{vert}$, where $\eta_\mathrm{hori} \in C_c^\infty(\langle j_0\rangle^\perp)$ is a horizontal cutoff function at transition length scale $l_\mathrm{hori}>0$ and  $\eta_\mathrm{vert} \in C_c^\infty(\langle j_0\rangle)$ is a vertical cutoff function at transition length scale $l_\mathrm{vert}>0$. 
    Specifically, write $Q_{j_0,\alpha} = Q_{j_0,\alpha}^\mathrm{hori} \oplus Q_{j_0}^\mathrm{vert}$. Choose $\eta_\mathrm{hori}\in C_c^\infty ( Q_{j_0,\alpha}^\mathrm{hori} )$ with $\eta_\mathrm{hori}(x) = 1$ whenever $\dist(x,\langle j_0\rangle^\perp \setminus Q_{j_0,\alpha}^\mathrm{hori}) \geq l_\mathrm{hori}$ and $\eta_\mathrm{vert}\in C_c^\infty ( Q_{j_0,\alpha}^\mathrm{vert} )$ with $\eta_\mathrm{vert}(y) = 1$ whenever $\dist(y,\langle j_0\rangle \setminus Q_{j_0,\alpha}^\mathrm{vert}) \geq l_\mathrm{vert}$.
    
\begin{figure}
  \begin{center}
    \includegraphics{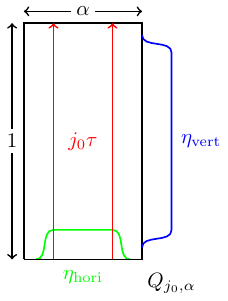}
    \end{center}
     \caption{A thin orthotope pointing in direction $j_0 = e_2$, with cutoff function $\eta(x_1,x_2) = \eta_\mathrm{hori}(x_1) \eta_\mathrm{vert}(x_2)$, and a divergence-free tangent measure $j_0 \tau$ supported on two vertical lines. }
     \label{fig: hori}.
\end{figure}

    Clearly (1) will be satisfied as long as $l_\mathrm{hori},l_\mathrm{vert}\leq \delta/2$.
    
      To check (2), first set $\tau_{\mathrm{hori}} = \eta_\mathrm{hori}\tau + (1-\eta_\mathrm{hori})U\in\M_+(Q_{j_0,\alpha})$.
    Then
    \begin{equation}
       \dive \left(j_0 \tau_\mathrm{hori}\right)  = \eta_\mathrm{hori} \underbrace{\dive(j_0\tau)}_{0} + (1-\eta_\mathrm{hori})\underbrace{\dive(j_0U)}_{0} +  \underbrace{j_0 \cdot \nabla \eta_{\mathrm{hori}}}_{0} (\tau - U) = 0.
    \end{equation}

   A direct calculation shows that $\eta \tau + (1-\eta)U= \eta_\mathrm{vert} \tau_\mathrm{hori} + (1-\eta_\mathrm{vert})U$. Multiplying with $j_0$ and taking the divergence yields
    \begin{equation}
    \begin{aligned}
    &\dive \left(j_0 \left(\eta \tau + (1-\eta)U\right)\right)\\
    = & \eta_\mathrm{vert} \underbrace{\dive(j_0\tau_\mathrm{hori})}_{0} + (1-\eta_\mathrm{vert})\underbrace{\dive(j_0U)}_{0} + j_0 \cdot \nabla \eta_\mathrm{vert} (\tau_\mathrm{hori} - U)\\
    = &j_0 \cdot \nabla \eta_\mathrm{vert} (\tau_\mathrm{hori} - U)
    \end{aligned}
    \end{equation}
    Note that unlike in the horizontal cutoff, the scalar product $j_0\cdot \nabla \eta_\mathrm{vert}$ is not zero but only depends on the vertical variable, whereas $\tau_\mathrm{hori} - U= (\tilde \tau_\mathrm{hori} - \tilde U)_x\otimes \Lm^1_y$ depends only on the horizontal variable $x$.

    To bound the KR-norm, consider a $1$-Lipschitz test function $\phi\in C_c^\infty(\R^n)$, which we split up into $\phi(x,y) = \bar \phi(y) + \tilde \phi(x,y)$, where $\bar \phi(y) = \phi(x_0,y)$, $\tilde \phi(x,y) = \phi(x,y) - \phi(x_0,y)$ for some $x_0\in R^\mathrm{hori}_{j_0,\alpha}$. Then
    \begin{equation}
    \begin{aligned}
    &|\langle\dive \left(j_0 \left(\eta \tau + (1-\eta)U\right)\right), \phi \rangle|\\
    = &|\langle j_0 \cdot \nabla \eta_\mathrm{vert} (\tau_\mathrm{hori} - U), \phi\rangle|\\
    \leq &|\langle j_0 \cdot \nabla \eta_\mathrm{vert} (\tau_\mathrm{hori} - U), \bar \phi\rangle| + |\langle j_0 \cdot \nabla \eta_\mathrm{vert} (\tau_\mathrm{hori} - U), \tilde \phi\rangle|.
    \end{aligned}
    \end{equation}

    The first pairing factorizes into horizontal and vertical terms:
        \begin{equation}
    \begin{aligned}
  &\langle j_0 \cdot \nabla \eta_\mathrm{vert} (\tau_\mathrm{hori} - U), \bar \phi\rangle \\
  =&\left(\int_{Q_{j_0}^\mathrm{vert}} j_0 \cdot \nabla \eta_\mathrm{vert}(y) \bar \phi(y)\,dy\right) \left(\tilde \tau_\mathrm{hori} - U\right)(Q_{j_0,\alpha}^\mathrm{hori})\\
  = & -\left(\int_{Q_{j_0}^\mathrm{vert}}\eta_\mathrm{vert}(y)  j_0 \cdot \nabla \bar \phi(y)\,dy\right) \left( \tau_\mathrm{hori}(Q_{j_0,\alpha}) - 1\right).
    \end{aligned}
    \end{equation}
    The absolute value of the vertical integral is bounded by $|j_0|$ since $\bar \phi$ is $1$-Lipschitz and the length of $Q_{j_0}^\mathrm{vert}$ is $1$. The horizontal difference tends to zero as $l_\mathrm{hori}\to 0$ by dominated convergence, since $\tau(\partial Q_{j_0,\alpha})  = 0.$ Thus the absolute value of the paring is at most $\alpha|j_0|$ as long as $l_\mathrm{hori}$ is small enough.

    We turn our attention to the second pairing and use the estimate $|\tilde \phi|\leq C\alpha$ in $Q_{j_0,\alpha}$, since $\tilde \phi$ is $1$-Lipschitz and zero on a vertical line segment and the diameter of $Q_{j_0,\alpha}^\mathrm{hori}$ is bounded by $C\alpha$:

            \begin{equation}
    \begin{aligned}
  &|\langle j_0 \cdot \nabla \eta_\mathrm{vert} (\tau_\mathrm{hori} - U), \tilde  \phi\rangle| \\
  \leq&\left( \int_{Q_{j_0}^\mathrm{vert}} |j_0||\nabla \eta_\mathrm{vert}(y)|\,dy\right) (\tilde \tau_\mathrm{hori} +  \tilde U)(Q_{j_0,\alpha}^\mathrm{hori})C\alpha\\
  \leq & |j_0|4C\alpha,
    \end{aligned}
    \end{equation}
    finishing the proof.
\end{proof}

Note that in the proof above the bound on the KR-norm of the divergence depends only on $\alpha$ and $l_\mathrm{hori}$ and not at all on $l_\mathrm{vert}$.

We finally prove Proposition \ref{prop: lower bound tangent measure}.

\begin{proof}[Proof of Proposition \ref{prop: lower bound tangent measure}]
    Before we begin, let us note that if $\dive \iota J^\eps = 0$ and $J^\eps = \calJ^\eps_{j_0}$ near $\partial Q_{j_0,\alpha}$, then $(m^\eps/\beta^\eps,J^\eps)$ are competitors to the cell problem for $j_0$ in the orthotope $Q_{j_0,\alpha}$, where $\beta^\eps = \sum_{x\in Q_{j_0,\alpha}} m^\eps(x) \to 1$. By definition of $f$ we then have
    \begin{equation}
        f(j_0) \leq \liminf_{\eps\to 0} F_\eps(m^\eps/\beta^\eps, J^\eps, Q_{j_0,\alpha}) =\liminf_{\eps\to 0} F_\eps(m^\eps, J^\eps, Q_{j_0,\alpha}),
    \end{equation}
    which is the claim.

    The proof consists of modifying the flow $J^\eps$ in two steps in order to ensure the boundary conditions and the divergence constraint respectively. In the third step we modify $m^\eps$ as well in order to bound the energy.

    \emph{Step 1: Boundary values. }
    Fix $\delta>0$. Let $\eta = \eta_\delta \in C_c^\infty(Q_{j_0,\alpha})$ be a cutoff function as defined in Lemma \ref{lemma: cutoff}. 
    Define $\tilde J^\eps = \eta \pentagram J^\eps + (1-\eta)\pentagram \calJ^\eps_{j_0} \in \R^{\cEe}_a$, 
    where $\pentagram$ is given by 
    \begin{align*}
    (\eta\pentagram J) (x,y)=\frac{\eta(x)+\eta(y)}{2}J(x,y),
\end{align*}
see Section \ref{appendix}.
    
    By our choice of cutoff function, we have
    \begin{equation}
        \lim_{\delta \to 0} \lim_{\eps\to 0}|\iota\tilde J^\eps - \iota J^\eps|(Q_{j_0,\alpha}) = 0,
    \end{equation}
    and by Lemma \ref{lemma: star}(2)
\begin{equation}
   \dive\iota  \tilde J^\eps \weakstar \dive(j_0(\eta\tau + (1-\eta)U))= \nabla \eta \cdot (j_0\tau - j_0U). 
\end{equation}
By the uniform boundedness principle \cite[Theorem 2.11]{rudin1991functional} and Riesz theorem \cite[Theorem 1.54]{ambrosio2000functions}
    \begin{equation}
        \sup_{\delta>0} \limsup_{\eps\to 0} |\dive \iota \tilde J^\eps|(Q_{j_0,\alpha}) < \infty.
    \end{equation}

By the compact embedding of narrow convergence on the compact set $Q_{j_0,\alpha}$ into $\KR$-convergence (Lemma \ref{lemma: KR}), we infer from the narrow convergence that
\begin{equation}
    \lim_{\eps \to 0} \|\dive \iota \tilde J^\eps - \dive(j_0(\eta\tau + (1-\eta)U)\|_\KR = 0
\end{equation}
for every $\delta >0$. Using Lemma \ref{lemma: cutoff} and the triangle inequality yields
    \begin{equation}
        \sup_{\delta>0} \limsup_{\eps\to 0} \|\dive \iota\tilde J^\eps\|_{\KR} \leq C\alpha|j_0|.
    \end{equation}
    
    \emph{Step 2: Corrector vector field. }
    In order to get rid of the divergence of $\tilde J^\eps$ we define $\hat J^\eps$ to be
    \begin{align*}
        \hat J^\eps=\tilde J^\eps+K^\eps,
    \end{align*}
    where $K^\eps\in\R^{\cEe}_a$ is a corrector vector field satisfying
    \begin{align*}
        \begin{cases}
            &\dive \iota K^\eps =-\dive \iota \tilde J^\eps\\
            &|\iota K^\eps|(Q_{j_0,\alpha})\leq C\|\dive \iota\tilde J^\eps\|_{\KR} + C\eps|\dive\iota \tilde J_\eps|(Q_{j_0,\alpha})\\
            &K^\eps=0 \text{ whenever} \dist([x,y],\del Q_{j_0,\alpha})\leq\dist(\supp(\eta),\del Q_{j_0,\alpha})-C\eps.
        \end{cases}
    \end{align*}

    To show the existence of $K^\eps$, first find a $W_1$-optimal coupling $\gamma\in \calM_+(Q_{j_0,\alpha}\times Q_{j_0,\alpha})$ between the positive measures of equal mass $(-\dive \iota\tilde J^\eps)_-$ and $(-\dive \iota\tilde J^\eps)_+$, so that $\langle \gamma, |y-x|\rangle =  \|\dive \iota \tilde J^\eps\|_{\KR}$.

Now for every pair $(x,y)\in \supp(\gamma) \subseteq \supp(\eta)\times \supp(\eta) 
\cap \cXe \times \cXe$, find a path $P(x,y)$ in $(\cXe,\cEe)$ of length at most $C(R)|y-x| + C(R)\eps$, which does not stray more than $C(R)\eps$ from the line segment $[x,y]$. These paths exist by Assumption \ref{ass: graph}.

Finally, define $K^\eps = \sum_{(x,y)\in\supp \gamma} \gamma(x,y) J_{P(x,y)}$. Since $\dive \iota J_{P(x,y)} = \delta_y - \delta_x$ and $\gamma$ is a coupling, we have $\dive\iota K^\eps = (-\dive \iota \tilde J^\eps)_+ - (-\dive\iota  \tilde J^\eps)_- = -\dive\iota  \tilde J^\eps$.

The previous estimates induce the following crucial properties of $\hat J^\eps$:
\begin{align*}
    \begin{cases}
        \limsup_{\delta\to0}\limsup_{\eps\to0}|\iota\hat J^\eps-\iota J^\eps|(Q_{j_0,\alpha})\leq C \alpha|j_0|,\\
        \dive \iota \hat J^\eps=0,\\
        \hat J^\eps(x,y)=\calJ^\eps_{j_0}(x,y) \text{ whenever} \dist([x,y],\del Q_{j_0,\alpha})\leq\dist(\supp(\eta),\del Q_{j_0,\alpha})-C\eps.
    \end{cases}
\end{align*}
    \emph{Step 3: Bound on the energy. }
    Finally we need to bound the energy. For this we replace $m^\eps$ by $\hat m^\eps$ defined as
    \begin{align*}
        \hat m^\eps(x)=\left(1-\frac{|\iota \hat J^\eps-\iota J^\eps|(Q_{j_0,\alpha})}{|j_0|}\right)m^\eps(x)+\frac{\frac12\sum_{y\sim x}|\hat J^\eps-J^\eps|(x,y)|x-y|}{|j_0|}.
    \end{align*}
    Then for every $\xi>0$
    \begin{align*}
        F_\eps(\hat m^\eps,\hat J^\eps,Q_{j_0,\alpha})=&\sum_{(x,y)\in Q_{j_0,\alpha}}\sigma_\omega(x,y)|x-y|^2\frac{|\hat J^\eps|^2(x,y)}{\theta(\hat m^\eps(x),\hat m^\eps(y))}\\
        \leq &(1+\xi)\sum_{(x,y)\in Q_{j_0,\alpha}}\sigma_\omega(x,y)|x-y|^2\frac{|J^\eps|^2(x,y)}{\theta(\hat m^\eps(x),\hat m^\eps(y))}\\
        &+ \left(1+\frac1\xi\right)\sum_{(x,y)\in Q_{j_0,\alpha}}\sigma_\omega(x,y)|x-y|^2\frac{|\hat J^\eps-J^\eps|^2(x,y)}{\theta(\hat m^\eps(x),\hat m^\eps(y))}\\
         \leq &(1+\xi)(1+2C\alpha)\sum_{(x,y)\in Q_{j_0,\alpha}}\sigma_\omega(x,y)|x-y|^2\frac{|J^\eps|^2(x,y)}{\theta( m^\eps(x),m^\eps(y))} \\
         &+ \left(1+\frac1\xi\right)|j_0|\sum_{(x,y)\in Q_{j_0,\alpha}}\sigma_\omega(x,y)|x-y|^2\frac{|\hat J^\eps-J^\eps|^2(x,y)}{|x-y||\hat J^\eps-J^\eps|(x,y)}.
    \end{align*}
    Consequently
    \begin{align*}
        F_\eps(\hat m^\eps,\hat J^\eps,Q_{j_0,\alpha})\leq &(1+\xi)(1+2C\alpha)F_\eps(m^\eps,J^\eps,Q_{j_0,\alpha}) + C\left(1+\frac1\xi\right)|j_0||\iota \hat J^\eps-\iota J^\eps|(Q_{j_0,\alpha})\\
        \leq& (1+\xi)(1+2C\alpha)F_\eps(m^\eps,J^\eps,Q_{j_0,\alpha}) + C\alpha\left(1+\frac1\xi\right)|j_0|^2.
    \end{align*}
    If $F_\eps(m^\eps,J^\eps,Q_{j_0,\alpha})\gg |j_0|^2$, then there is nothing to show. If not we optimize in $\xi$ and obtain $\xi=C\sqrt\alpha$. Rearranging the inequality we find for $\alpha$ small enough
     \begin{align*}
       F_\eps(m^\eps,J^\eps,Q_{j_0,\alpha})
        \geq&  (1-C\sqrt\alpha)F_\eps(\hat m^\eps,\hat J^\eps,Q_{j_0,\alpha}) - C\sqrt\alpha|j_0|^2.
    \end{align*}
    Taking the $\liminf$ yields
    \begin{align*}
        \liminf_{\eps\to0}  F_\eps(m^\eps,J^\eps,Q_{j_0,\alpha})
        \geq(1-C\sqrt\alpha)f(j_0)-C\sqrt\alpha|j_0|^2
        \geq f(j_0)-C\sqrt\alpha|j_0|^2,
    \end{align*}
    where we used the lower bound in Lemma \ref{lemma: f}(2).
\end{proof}

\subsection{Proof of the upper bound}
If $A(\rho,j)=\infty$, we only need to show the existence of $m^\eps\in \calP(\cXe)$ with $m^\eps \weakstar \rho$, which is classical, and the existence of $J^\eps \in L^1((0,T);\R^{\cEe}_a)$ with $\iota J^\eps \weakstar j$, which do not have to solve the continuity equation. First, approximate $j$ by an $L^1$-vector field which is piecewise constant on space-time cubes. Then approximate this piecewise-constant vector field with a piecewise-constant uniform representative $J^\eps$ from Definition \ref{def: uniform}. We note here that if $(\rho,j)$ has infinite action but solves the continuity equation distributionally, it is possible to find pairs $(m^\eps,J^\eps)$ which also solve the continuity equation and have finite action.

From now on, we can assume that $A(\rho,j)<\infty$, i.e. that $(\rho,j)$ is a finite action curve.

Our strategy consists of discretizing $(\rho,j)$ in time and space. We then subdivide each time interval $[t_k,t_{k+1}]$ into a flow phase $[t_k,t_{k+1}-\eta h]$ and a maintenance phase $[t_{k+1}-\eta h, t_{k+1}]$.

    During the flow phase, $J^\eps$ will be constant in time, and $m^\eps$ affine. At the end of each flow phase, the mass associated with every cube $z\in \delta\Z^n$ will be the same as the mass associated with $z$ at the start of the next flow phase.

\emph{Step 0:} Without loss of generality, $(\rho,j)$ is a smooth finite action curve with compact support in $[-M/2,M/2]^n$, $M> 0$. In order to approximate a general finite action curve with a smooth compactly supported one, first mollify in time, then orthogonally project onto a large ball in space, then mollify in space. See e.g. \cite[Theorem 8.2.1]{ags} for a detailed argument.

\emph{Step 1:} Discretize the continuity equation in time and space by partitioning $[0,T]$ into intervals $[t_k,t_{k+1}]$ of length $h>0$ and partitioning $\R^n$ into closed cubes $Q(z,\delta)$ of size $\delta>0$ with centers $z \in \delta  \Z^n$, defining the discrete quantities 
\[
\rho_{t_k}^z = \int_{Q(z,\delta)}\rho(t_k,x)\,dx \lesssim \delta^n,
\]
\[
j_{t_k,t_{k+1}}^{z,z'} = \fint_{t_k}^{t_{k+1}}\int_{\partial Q(z,\delta)\cap \partial Q(z',\delta)} j(t,x)\cdot \frac{z'-z}{\delta}\,d\Hm^{n-1}(x)\,dt \lesssim \delta^{n-1},
\]
so that 
\begin{align*}
    \frac{\rho_{t_{k+1}}^z-\rho_{t_k}^z}{h}+\sum_{z'\sim z}j_{t_k,t_{k+1}}^{z,z'}=0 \text{ for all }z\in\delta\Z^n, k= 0,\ldots,T/h-1.
    \end{align*}

\emph{Step 2:} Define flows $v^\eps_{t_k,t_{k+1}}$ on $(\eps\Z^n,n.n.)$ for $\eps \ll \delta$, $\frac{\delta}{\eps}\in 2\N + 1$, through
\[
v^\eps_{t_k,t_{k+1}}(a,b) =
\begin{cases}
\frac{\eps^{n-1}}{\delta^{n-1}} j^z_{t_k,t_{k+1}}\cdot\frac{b-a}{|b-a|}&\text{ if }[a,b]\subseteq Q(z,\delta)\\
\frac{\eps^{n-1}}{\delta^{n-1}} j^{z,z'}_{t_k,t_{k+1}} & \text{ if } a\in Q(z,\delta), b\in Q(z',\delta), z\neq z',
\end{cases}
\]
    where $j_{t_k,t_{k+1}}^{z}= \frac12(j_{t_k,t_{k+1}}^{z,z+\delta e_i} + j_{t_k,t_{k+1}}^{z-\delta e_i,z})_{i=1}^n\in  \R^n$.

Note that all edges are covered by the above two cases, since $\eps \Z^n$ and $\partial Q(z,\delta)$ are disjoint for every $z\in\delta \Z^n$, since the edge length of the cube is an odd multiple of $\eps$. Calculating the divergence of $v^\eps_{t_k,t_{k+1}}$ yields the useful properties
\begin{equation}\label{eq: v CE}
    \sum_{a\in \eps\Z^n \cap Q(z,\delta)} \sum_{b \sim a} v^\eps_{t_k,t_{k+1}}(a,b) = \sum_{z'\sim z}j_{t_k,t_{k+1}}^{z,z'}\end{equation}
     and 
     \begin{equation}\label{eq: v divergence}
    \|\dive \iota v^\eps_{t_k,t_{k+1}}\|_{l^\infty} \leq  \frac{n}{2} \mathrm{Lip}(j)\delta \eps^{n-1}.
\end{equation}

\emph{Step 3:} We define the backbone of our optimal flow piecewise constant in the reduced time intervals $I_k = [t_k,t_{k+1} -\eta h]$, where $\eta\in (0,1)$ is small but not too small. The backbone flow is $\tilde J^\eps_{t} : \bigcup_{k=0}^{T/h - 1}I_k \to \R^{\cEoe}_a$,
\begin{equation}
    \tilde J^\eps_{t} = \frac1{1-\eta}\phi_\eps^\# v^\eps_{t_k,t_{k+1}},\quad t\in I_k,
\end{equation}
where $\phi_\eps:(\eps \Z^n,n.n.) \to (\cXe, \cEe)$ is the graph homomorphism from Lemma \ref{lemma: homo}.

We note that $\iota \tilde J^\eps_t$ has nonzero divergence near the boundaries $\partial Q_{z_0,\delta}$. In order to satisfy the continuity equation, we have to deposit mass at the divergence sites at the beginning of each flow phase, which will then change affinely throughout the flow phase. 

The depot vertices are defined as $\phi_\eps(D_{\delta, \eps})\subset \cXe$, where
\begin{equation}
D_{\delta,\eps} = \left\{a\in \eps\Z^n\,:\, \dist\left(a, \bigcup_{z\in \delta \Z^n} \partial Q(z,\delta)\right) \leq \frac{\eps}{2} \right\},
\end{equation}
see Figure \ref{fig:depot}.

We define the depot mass distributions $m^\eps_{\mathrm{depot}, z} :\bigcup_k I_k \to  [0,\infty)^{\cXe}$, supported on the depots $\phi_\eps(D_{\delta,\eps})$, as 
\begin{equation}
m^\eps_{\mathrm{depot},z}(t,\cdot) = \phi_\eps^\# \left( \left(\alpha \delta \eps^{n-1} \mathds{1}_{D_{\delta,\eps}} - (t-t_k) \dive v^\eps_t \right) \mathds{1}_{Q(z,\delta)}\right), \text{ for }t\in I_k.
\end{equation}

\begin{figure}[h!]
  \centering
  \begin{subfigure}[b]{0.4\linewidth}
    \includegraphics[width=\linewidth]{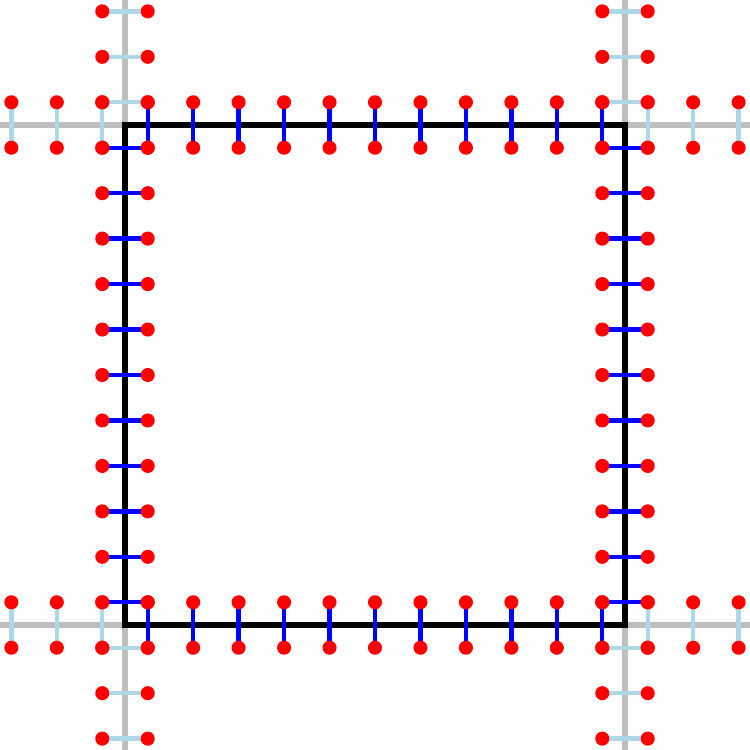}
    \caption{}
  \end{subfigure}
  \hspace{1cm}
  \begin{subfigure}[b]{0.4\linewidth}
    \includegraphics[width=\linewidth]{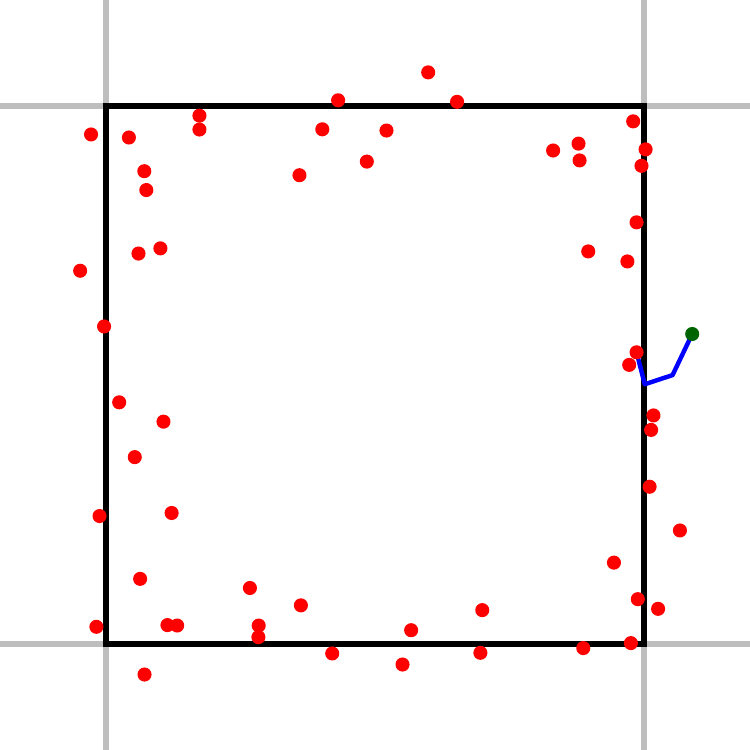}
    \caption{}
  \end{subfigure}
  \caption{(A) All depot vertices $D_{\delta,\eps}$ in $\eps\Z^n$, with interfacial edges. (B) Depot vertices $\phi_\eps(D_{\delta,\eps})$ in the random graph (only the depots associated with a single cube $Q_{z_0,\delta}$), with a single interfacial path.}
  \label{fig:depot}
\end{figure}

Here we require $\alpha\geq h \frac{n}{2} \Lip(j)$, so that $m^\eps_{\mathrm{depot},z} \geq 0$ for all $x\in\cXe$ and all times $t\in \bigcup_k I_k$ by \eqref{eq: v divergence}. In other words, depot masses are standardized at the beginning of each flow phase $t=t_k$ and then filled or emptied gradually throughout the flow phase.

We see that the pair $(\sum_{z \in \delta \Z^n \cap [-M,M]^n} m^\eps_{\mathrm{depot},z}, \tilde J^\eps)$ solves \eqref{eq: dcont} on $\bigcup_k I_k$ by our choice of $\partial_t m^\eps_{\mathrm{depot},z}(t,\cdot)$. The total mass of $\sum_{z \in \delta \Z^n \cap [-M,M]^n} m^\eps_{\mathrm{depot},z}$ is  bounded by the small number $n\alpha (2M)^n$.

Additionally, we have the identity
\begin{equation}\label{eq: time gaps}
    \sum_{x\in\cXe} (m^\eps_{\mathrm{depot},z}(t_{k+1}-\eta h,x) - m^\eps_{\mathrm{depot},z}(t_k,x)) = \rho^z_{t_{k+1}} -\rho^z_{t_k} 
\end{equation}
for all $k=0,\ldots,T/h-1$ and all $z\in \delta \Z^n$, which follows from \eqref{eq: v CE}.

\emph{Step 4:} Since $\tilde J^\eps_t$ is piecewise constant in time and equal to $\frac{1}{1-\eta}\calJ^\eps_{j^z_{t_k,t_{k+1}}}$ on the slightly smaller cube $q_z = Q(z,\delta - 3R\eps)$, we may replace $\tilde J^\eps_t$ by an optimal microstructure on all edges intersecting $q_z$ without incurring additional divergence!

More specifically, choose optimal pairs $(m^\eps_{z,t_k,t_{k+1}}, J^\eps_{z,t_k,t_{k+1}})$ in the cell formula \eqref{eq: cell} for $(\rho^z_{t_k}, \frac{1}{1-\eta}j^z_{t_k,t_{k+1}})$ in the cube $q_z$.

In particular, we have 
\begin{equation*}
\lim_{\eps \to 0} F(m^\eps_{z,t_k,t_{k+1}}, J^\eps_{z,t_k,t_{k+1}}, q_z) = \Lm^n(q_z)f\left(\frac{\frac{1}{1-\eta}j^z_{t_k,t_{k+1}}}{\rho^z_{t_k}}\right)\rho^z_{t_k}
\end{equation*}

We glue these flows together to create the global flow:
\begin{equation}
    J^\eps_t(x,y) = \begin{cases}
    J^\eps_{z,t_k,t_{k+1}}(x,y) &\text{ if }t\in I_k, (x,y) \cap q_z \neq \emptyset\\
    \tilde J^\eps_{t_k,t_{k+1}}(x,y) &\text{ if }t\in I_k, (x,y) \cap q_z = \emptyset\text{ for all }z\in \delta \Z^n.
    \end{cases}
\end{equation}

We also define the global mass distribution
\begin{equation}
m^\eps_t(x) = \sum_{z\in\delta \Z^n \cap [-M,M]^2} (m^\eps_{z,t_k,t_{k+1}}(x) + m^\eps_{\mathrm{depot}, z}(t,x)) + m^\eps_\mathrm{path}(x), \text{ for }t\in I_k,
\end{equation}
where $m^\eps_\mathrm{path}\in [0,\infty)^\cXe$ is an error term independent of time and will be specified later. Note that $(m^\eps, J^\eps)$ still solve \eqref{eq: dcont} on $\bigcup_k I_k$, since we created no additional time derivative and no additional divergence.

Using \eqref{eq: time gaps} and the fact that $\sum_{x\in\cXe} m^\eps_{z,t_k,t_{k+1}}(x) = \rho^z_{t_k}$, we arrive at
\begin{equation}\label{eq: masses}
  \sum_{x\in \cXe}  m^\eps_{z,t_k,t_{k+1}}(x) + m^\eps_{\mathrm{depot}, z}(t_{k+1}-\eta h,x) =     \sum_{x\in \cXe}  m^\eps_{z,t_{k+1},t_{k+2}}(x) + m^\eps_{\mathrm{depot}, z}(t_{k+1},x).
\end{equation}
In other words, the mass associated with $z$, which is located in $B(Q(z,\delta),R\eps)$, does not change between $t_{k+1}-\eta h$ and $t_{k+1}$! This will allow us to cheaply bridge the gap as long as $\delta^2 \ll \eta h^2$:

We define $m^\eps_\mathrm{path}\in [0,\infty)^\cXe$ as
\begin{equation}
    m^\eps_\mathrm{path}(x) = \begin{cases}
    0 &x\in q_z \text{ for any }z\in\delta \Z^n,\\
    \sum_{y\sim x}|y-x||\calJ^\eps(x,y)| & x\in [-M,M]^n\setminus \bigcup_{z\in\delta \Z^n} q_z
        \end{cases}
\end{equation}
The total mass of $m^\eps_\mathrm{path}$ is bounded by 
\begin{align*}
    \sum_{x\in\cXe}m^\eps_\mathrm{path}\leq |\iota\calJ^\eps|([-M,M]^n\setminus\bigcup_{z\in\delta\Z^n} q_z)\leq CM^n\eps.
\end{align*}

\emph{Step 5:} We fill in the gaps $[0,T]\setminus \bigcup_k I_k=\bigcup_{k=0}^{T/h-1}[t_{k+1}-\eta h, t_{k+1}]$. 
Using \eqref{eq: masses} we find for all $k=0,\ldots, T/h-1$
\begin{align*}
    W_\infty(m^\eps_{t_{k+1}-\eta h},m^\eps_{t_{k+1}})\leq 2\sqrt{n}\delta.
\end{align*}
According to the a-priori bound on the discrete action \cite[Lemma 3.3]{gladbach2020scaling}
there exist curves $(m^\eps_t,J^\eps_t)$ on the intervals $(t_{k+1}-\eta h,t_{k+1})$ with 
\begin{align*}
\int_{t_{k+1}-\eta h}^{t_{k+1}}F_\eps(m^\eps_t,J^\eps_t)\, dt
\leq C \frac{(W_\infty(m^\eps_{t_{k+1}-\eta h},m^\eps_{t_{k+1}})+\eps)^2}{\eta h} \leq C\frac{\delta^2}{\eta h}.
\end{align*}
If $\delta^2\ll \eta h^2$, the total contribution from $\frac{T}{ h}$ maintenance intervals is small.

\emph{Step 6: }
Let $t\in I_k$. We estimate 
\begin{align*}
    F_\eps(m^\eps_t,J^\eps_t)
    \leq&\sum_{z\in \delta \Z^n}\sum_{(x,y)\cap q_z\neq\emptyset}\sigma(x,y)|x-y|^2\frac{J^\eps_{z,t_k,t_{k+1}}(x,y)^2}{m^\eps_{z,t_k,t_{k+1}}(x)}\\
    & +\sum_{(x,y)\cap q_z=\emptyset}\sigma(x,y)|x-y|^2\frac{\tilde J^\eps_{t_k,t_{k+1}}(x,y)^2}{m^\eps_{\mathrm{path}}(x)}\\
    \leq &\sum_{z\in\delta\Z^n}F(m^\eps_{z,t_k,t_{k+1}}, J^\eps_{z,t_k,t_{k+1}}, q_z) + \sum_{(x,y)\cap q_z=\emptyset}\sigma(x,y)|x-y||\tilde J^\eps_{t_k,t_{k+1}}(x,y)|\|j\|_\infty\\    
    \leq &\sum_{z\in\delta\Z^n} F(m^\eps_{z,t_k,t_{k+1}}, J^\eps_{z,t_k,t_{k+1}}, q_z) +C |\iota\calJ^\eps|([-M,M]^n\setminus\bigcup_{z\in\delta\Z^n} q_z) \|j\|_\infty^2 \\
    \leq &\sum_{z\in\delta\Z^n} F(m^\eps_{z,t_k,t_{k+1}}, J^\eps_{z,t_k,t_{k+1}}, q_z) + C \eps \|j\|_\infty^2 M^n.
\end{align*}

The total action on $(0,T)$ can be estimated by
\begin{equation}
   \lim_{h\to 0}\lim_{\eta \to 0} \lim_{\delta \to 0} \lim_{\eps \to 0} \int_0^T F_\eps(m^\eps_t,J^\eps_t)\, dt \leq \int_0^T \int_{\R^n} f(\rho(t,x),j(t,x))\, dx \, dt.
\end{equation}

Taking a suitable diagonal sequence $h(\eps),\eta(\eps),\delta(\eps) \to 0$ yields the desired sequence $(m^\eps,\iota J^\eps) \weakstar (\rho,j)$ with $\lim_{\eps \to 0} F_\eps(m^\eps_t,J^\eps_t)\, dt = \int_0^T \int_{\R^n} f(\rho(t,x),j(t,x))\, dx \, dt$, where we use the lower bound to obtain equality. In addition, $\max_{t\in[0,1]} W_2(m^\eps_t,\rho_t) \to 0$.

\appendix
\section{}\label{appendix}

\begin{lemma}\label{lemma: star}
   Let $\eta$ be a function in $C^0(\R^n)$ and $J^\eps$ a sequence of vector fields in $\R^{\cEe}_a$. We define their product $\eta \pentagram J\in \R^{\cEe}_a$  by
\begin{align*}
    (\eta\pentagram J) (x,y)=\frac{\eta(x)+\eta(y)}{2}J(x,y).
\end{align*}
Then:
    \begin{enumerate}
       
      \item Modulation:
      Let $ \iota J^\eps\weakstar \nu$ vaguely. Then
      \begin{align*}
       \iota(\eta \pentagram J^\eps) \weakstar \eta \nu \text{ vaguely.}
      \end{align*}
      \item Convergence of divergence:
        Let $\eta\in C^1(\R^n)$. Let $J^\eps\weakstar \nu$ and $\dive \iota  J^\eps\weakstar \dive \nu$ vaguely. Then
\begin{align*}
    \dive \iota (\eta\pentagram J^\eps) \weakstar \dive (\eta \nu) \text{ vaguely.}
\end{align*}
    \end{enumerate}
    
\end{lemma}

\begin{proof}
    The first assertion is a standard quadrature argument. For the second, take a test function $\varphi\in C_c^0(\R^n)$. Then applying the fundamental theorem of calculus to $\eta$ on each line segment $[x,y]$ yields

\begin{align*}
    \langle \dive\iota(\eta \pentagram J^\eps), \varphi \rangle =& \sum_{x\in\cXe} \sum_{y\sim x} \frac{\eta(y) + \eta(x)}{2} J^\eps(x,y) \varphi(x) \\
     = &  \sum_{x\in\cXe} \sum_{y\sim x} \eta(x) J^\eps(x,y) \varphi(x)\\
     &+\frac12 \sum_{x\in \cXe} \sum_{y\sim x}  (\eta (y)-\eta(x)) J^\eps(x,y)\varphi(x)\\
     = & \langle \dive \iota J^\eps, \eta\varphi \rangle\\
     &+\frac12 \sum_{x\in \cXe} \sum_{y\sim x} \int_{[x,y]} \nabla \eta (z) \cdot \frac{y-x}{|y-x|}J^\eps(x,y)\varphi(x)\,d\Hm^1(z)\\
     =& \langle \dive \iota J^\eps, \eta\varphi \rangle\\ 
     & +\sum_{(x,y)\in \cEe} \int_{[x,y]}  \frac{\varphi(x) + \varphi(y)}{2} \nabla \eta (z) \cdot \frac{y-x}{|y-x|}J^\eps(x,y)\,d\Hm^1(z)\\
     = & \langle \dive \iota J^\eps, \eta\varphi \rangle
 +\langle \iota(\varphi \pentagram J^\eps), \nabla \eta \rangle \longrightarrow_{\eps\to 0} \langle \dive \nu, \eta\varphi \rangle
+ \langle \varphi \nu, \nabla \eta\rangle\\
=& \langle \dive \nu, \eta\varphi \rangle + \langle \nabla \eta \cdot \nu, \varphi \rangle= \langle \dive (\eta\nu),\varphi \rangle,
\end{align*}
where we used (1).
   
\end{proof}

\begin{lemma}\label{lemma: KR}
Let  $\sigma\in \M(\R^n)$ and $(\sigma_n)\subset\M(\R^n)$ such that $\sigma_n\weakstar\sigma$ as $n\to\infty$ and $\ip{\sigma_n,1}=0$ for all $n\in\N$.
Assume there exists a compact set $K\subset\R^n$ such that $\supp\sigma_n\subset K$ for all $n\in\N$.
   Then $\|\sigma_n-\sigma\|_{\KR}\to0$ as $n\to\infty$.
\end{lemma}
\begin{proof}
    Let $\phi_n\in C_c(\R^n)$ be $1$-Lipschitz with $\|\sigma_n-\sigma\|_{\KR}=\ip{\sigma_n-\sigma,\phi_n}$. Since $\ip{\sigma_n,1}=0$ and $\supp\sigma_n\subset K$, we may assume that $\phi_n(0)=0$ by adding constants. By the classical Arzel\`a-Ascoli theorem, there exists a $\phi\in C_b(\R^n)$ such that $\phi_n\to\phi$ uniformly in $K$. We rewrite
    \begin{align*}
        \ip{\sigma_n-\sigma,\phi_n}=\ip{\sigma_n-\sigma,\phi}+\ip{\sigma_n-\sigma,\phi_n-\phi}
    \end{align*}
    and notice that by the narrow convergence of $\sigma_n$ we have $\ip{\sigma_n-\sigma,\phi}\to0$. The second term we bound by employing the total variation
    \begin{align*}
        |\ip{\sigma_n-\sigma,\phi_n-\phi}|\leq (|\sigma_n|(K)+|\sigma|(K))\sup_{x\in K}|\phi_n(x)-\phi(x)|.
    \end{align*}
    By Riesz theorem \cite[Theorem 1.54]{ambrosio2000functions} and the uniform boundedness principle \cite[Theorem 2.11]{rudin1991functional} we have that $|\sigma_n|(K)+|\sigma|(K)$ is uniformly bounded in $n$, which implies that $\ip{\sigma_n-\sigma,\phi_n-\phi}\to 0$ since $\phi_n\to\phi$ uniformly in $K$. All together, this implies that  $\|\sigma_n-\sigma\|_{\KR}\to0$ as $n\to\infty$.
\end{proof}

\bibliographystyle{siam}
{\bibliography{stoch_hom_GK}}
\end{document}